\documentclass[final,leqno]{siamltex}
\usepackage{amsfonts}
\usepackage{latexsym}
\usepackage{amsmath}
\usepackage{amssymb}
\usepackage{amscd}
\usepackage{epsfig}
\usepackage{graphicx, subfigure}
\usepackage{color}
\addtocounter{MaxMatrixCols}{4}

\title{A Linear-size Conversion of HCP to 3HCP}
\author{V. Ejov\thanks{Flinders University, Australia, vladimir.ejov@flinders.edu.au} \and M. Haythorpe\thanks{Flinders University, Australia, michael.haythorpe@flinders.edu.au} \and S. Rossomakhine\thanks{Flinders University,
Australia, serguei.rossomakhine@flinders.edu.au}}

\begin{document}

\maketitle
\begin{abstract}We provide an algorithm that converts any instance of the Hamiltonian cycle problem (HCP) into a cubic instance of HCP (3HCP), and prove that the input size of the new instance is only a linear function of that of the original instance. This is achieved by first considering various restrictions of HCP. Known conversions from directed HCP to undirected HCP, and sub-cubic HCP to cubic HCP are given. We introduce a subgraph called a 4-gate and show that it may be used to convert sub-quartic HCP into sub-cubic HCP. We further generalise this idea by first introducing the 5-gate, and then the $s$-gate for any $s \geq 4$. We prove that these subgraphs may be used to convert general instances of HCP into cubic HCP instances, where the input size of the converted instance is a quadratic function of that of the original instance. This result improves upon the previously best known approach which results in cubic growth in the size of the instance. We further prove that the quadratic function is reduced to a linear function if the maximum initial degree is bounded above by a constant. Motivated by this result, we describe an algorithm to convert general HCP to HCP of bounded degree and prove that it results in only linear growth. All of the above results are then used in the proof that any instance of HCP may be converted to an equivalent instance 3HCP with only linear growth in the input size.
\end{abstract}



\pagestyle{myheadings}
\thispagestyle{plain}

\section{Introduction}

The Hamiltonian cycle problem (HCP) is a famous graph theoretic decision problem that can be described simply: given a graph $\Gamma$, determine whether it contains any simple cycles containing all vertices in the graph, or
not. Although HCP has been studied for over 250 years (Euler was studying the Knights Tour Problem, a special case of HCP, as far back as 1759), it has seen a recent resurgence in interest due to its placement in the
NP-complete\footnote{For an in-depth study of NP-complete problems, the interested reader is referred to Garey and Johnson \cite{gareyjohnson}.} list of problems. Indeed, the reduction from Vertex Cover to HCP was among the
first NP-complete reductions published \cite{karp}.

Throughout this paper, when referring to HCP in its traditional form, we will use the expression {\em general HCP}. An {\em instance} of general HCP takes the form of a simple graph, which is defined by its vertices and (directed) edges. We refer to the sum of the number of vertices and the number of edges of an instance as the {\em input size} of the instance. So if a graph contains $N$ vertices and $e$ edges, the input size is $N + e$. However, any meaningful instance of HCP will have at least as many edges as vertices, or else it is trivially non-Hamiltonian. For this reason, when talking about the order of the input size, it suffices to merely consider the number of edges in the graph.

Although the definition of general HCP permits any simple graph as an instance, a natural idea is to adopt restricted definitions of HCP in which only graphs that satisfy certain properties are to be considered. For some such restrictions, HCP is known to remain NP-complete. The first of these was proved by Karp \cite{karp}, who showed that any general HCP instance can be converted to an equivalent instance which only contains undirected edges. The input size of the new instance is a linear function of that of the original instance. We will use the expression {\em undirected HCP} to describe the restricted version of HCP where only undirected instances are permitted. For the sake of completeness, we include the reduction here.

{\bf Converting General HCP to Undirected HCP \cite{karp}}

Suppose we have a graph containing $N$ vertices, which forms an instance of general HCP. We can produce a new graph containing $3N$ vertices, and add edges to it using the following algorithm.

{\underline{General HCP to Undirected HCP Conversion Procedure}}

\begin{itemize}\item Add edges $(3i-1,3i-2)$ and $(3i-1,3i)$ for all $i = 1, \hdots, N$.
\item For each (directed) edge $(i,j)$ in the original graph, add edge $(3i,3j-2)$.\end{itemize}

In the above procedure, a new graph instance is constructed from scratch. For convenience, however, for the remainder of this manuscript we will think of such procedures as having replaced certain components of a graph with new components, whose constructions depend upon the
components they are replacing. In the above procedure, each vertex is being replaced by three vertices, the second of which is linked to the other two by undirected edges. Then, for each vertex, incoming edges are replaced by undirected edges adjacent to the first vertex, and outgoing edges are replaced by undirected edges adjacent to the third vertex. This can be seen in Figure \ref{fig-directed}. It is also important to note here that if the original instance has maximum in-degree $r$ and maximum out-degree $s$, the new undirected instance will have maximum degree of $\max(r,s) + 1$.

\begin{figure}[h!]
\centering\includegraphics[scale=0.5]{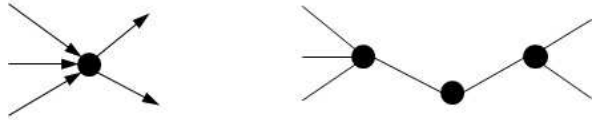}
\caption{A vertex with adjacent directed edges, and the corresponding undirected subgraph which replaces it.\label{fig-directed}}
\end{figure}

It is easy to see that the undirected HCP instance resulting from the General HCP to Undirected HCP Conversion Procedure is equivalent to the original instance. Consider one of the sets of three vertices, and call them $a$, $b$ and $c$. Suppose a Hamiltonian cycle visits vertex $a$. Vertex $a$ is adjacent to the degree 2 vertex $b$, so the latter must be visited immediately, and the cycle will then continue to vertex $c$. From there, the only edges present correspond to outgoing edges in the original graph, and each of these lead to the first vertex in another block of three vertices. Inductively, this process continues through the entire graph, and it is clear that a Hamiltonian cycle may only exist in the undirected instance if one existed in the original instance. Indeed, there is a 1-1 correspondence between Hamiltonian cycles in the original instance and Hamiltonian cycles in the equivalent undirected graph.

Throughout this manuscript, we will refer to conversions as being {\em polynomially-growing}, depending on the degree of the polynomial that describes the new input size as a function of the old input size. For example, the above
procedure describes a {\em linearly-growing conversion}, because the resultant graph has input size which is a linear function of the original input size. Specifically, in the above conversion, if the original graph had $N$ vertices and $e$ (directed) edges, the new undirected graph will contain $3N$ vertices, and $e + 2N$ undirected edges.

In addition to the conversion from general HCP to undirected HCP, it was proved by Garey et al \cite{tarjan} that even if HCP is restricted to only instances that are undirected, cubic, planar and 3-connected, the problem is still
NP-complete. Of course, this implies that one may restrict HCP instances to any subset of those four conditions and the problem will remain NP-complete. In this manuscript we will be primarily interested in cubic, undirected HCP
instances. We will refer to HCP restricted to such instances as {\em cubic HCP} (note that the undirected condition is implied in this title). Cubic HCP is a widely studied problem in its own right, with Barnette's
conjecture \cite{barnette} that every bipartite, planar, 3-connected cubic graph is Hamiltonian still an open conjecture. More recently, Eppstein \cite{eppstein} conjectured that undirected cubic graphs with $N$ vertices have at
most $2^{N/3}$ Hamiltonian cycles, with Gebauer \cite{gebauer} providing the best proven bound to date of approximately $1.276^N$. Eppstein \cite{eppstein} also provided an algorithm for finding Hamiltonian cycles in cubic graphs. Another open conjecture by Filar et al \cite{conjecturepaper} is that almost all non-Hamiltonian cubic graphs are bridge graphs, and may therefore be easily detected.

The result of Garey et al \cite{tarjan} proves it is possible to convert general HCP instances into cubic HCP instances, allowing us to take advantage of the special structure of cubic graphs. However, in practice, using their approach is inefficient, as the reduction in \cite{tarjan} is not from general HCP, but rather from boolean satisfiability (SAT) in conjunctive normal form with clauses of size 3 (3SAT). Currently, the best known conversion from general HCP to SAT is a cubically-growing conversion \cite{ariadne}, and the conversion of SAT to 3SAT is a linearly-growing conversion \cite{karp}. Converting from 3SAT to cubic, planar, 3-connected HCP using the method by Garey et al requires a quadratically-growing conversion; however, if we drop the requirement of planarity, the conversion reduces to linearly-growing. So, to convert a general HCP instance into a cubic HCP instance via the approach in \cite{tarjan}, we must first convert to SAT, then to 3SAT, and finally to cubic HCP, which results in a cubically-growing conversion.

In this manuscript, we describe a new approach to convert directly (that is, we remain within the scope of HCP for the entirety of the conversion) from general HCP to cubic HCP, using a linearly-growing conversion. We begin by revising the well-known conversion for sub-cubic (undirected) HCP to cubic HCP. We then give conversions, to sub-cubic HCP, for the special cases of undirected HCP where the maximum degree is 4 or 5. We follow this with a quadratically-growing conversion for general HCP to cubic HCP, and show that this approach constitutes a linearly-growing conversion if the maximum degree is bounded above by a constant. We also show that this approach can be applied to directed graphs, if desired. Finally we provide a linearly-growing conversion from general HCP to HCP with bounded in-degree and out-degree (for a sufficiently large upper bound), and using the previous results we conclude that general HCP may be reduced to cubic HCP via a linearly-growing conversion. We conclude with some examples of the savings to be gained by using this approach, compared with the approach in \cite{tarjan}.

\section{Converting Sub-cubic HCP to Cubic HCP} 

Consider an undirected graph with maximum degree 3. We refer to HCP restricted to such instances as {\em sub-cubic HCP}. Then there is a simple procedure to convert sub-cubic HCP to cubic HCP.

{\underline{Sub-cubic HCP to Cubic HCP Conversion Procedure}}

\begin{itemize}\item If the sub-cubic instance has any degree 1 vertices, the graph is non-Hamiltonian, and may be replaced by any non-Hamiltonian cubic graph (such as the Petersen graph \cite{holton}).
\item Otherwise, replace any degree 2 vertices with a diamond subgraph, as shown in Figure \ref{fig-degree2}.\end{itemize}

\begin{figure}[h!]
\centering\includegraphics[scale=0.5]{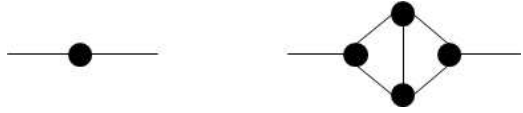}
\caption{A degree 2 vertex, and the corresponding cubic subgraph which replaces it.\label{fig-degree2}}
\end{figure}

It is clear that the resultant graph is cubic. To see that the Hamiltonicity has not been altered, it suffices to recognise that an introduced diamond, once entered, must be fully traversed before departing, as it will be
impossible to enter it again. Then this subgraph functions exactly the same as the vertex it replaced, and therefore the new cubic instance contains Hamiltonian cycles if and only if the original instance did.

It is also clear that the above procedure constitutes a linearly-growing conversion. Even in the worst case, where all vertices are of degree 2, the number of vertices is only quadrupled and the number of edges is only sextupled.

It is worth noting that the diamond subgraphs used in the above conversion can be traversed in either of two different ways. Therefore, there is not a 1-1 relationship between the Hamiltonian cycles in the sub-cubic instance, and the Hamiltonian cycles in the converted instance. Indeed, if the sub-cubic instance contains $k$ vertices of degree 2, then for each Hamiltonian cycle in the sub-cubic instance there are $2^k$ different Hamiltonian cycles in the
cubic instance. The 1-many relationship is unfortunately unavoidable in general, as it is a known result that undirected cubic Hamiltonian graphs must contain at least three Hamiltonian cycles \cite{tutte}, but sub-cubic graphs may contain any number of Hamiltonian cycles, i.e. 1 or 2 in particular.


\section{Undirected Graphs with Maximum Degree 4 and 5}

Consider an undirected graph with maximum degree 4. We refer to HCP restricted to such instances as {\em sub-quartic HCP}. Then there is a simple procedure to convert sub-quartic HCP to sub-cubic HCP. First we define a subgraph
which we call a {\em 4-gate}, as displayed in Figure \ref{fig-4gate}. Note that there are four edges, indicated by dashed lines in Figure \ref{fig-4gate}, by which the 4-gate may be entered or exited. We will refer to these four
edges as {\em external edges}.

\begin{figure}[h!]
\centering\includegraphics[scale=0.625]{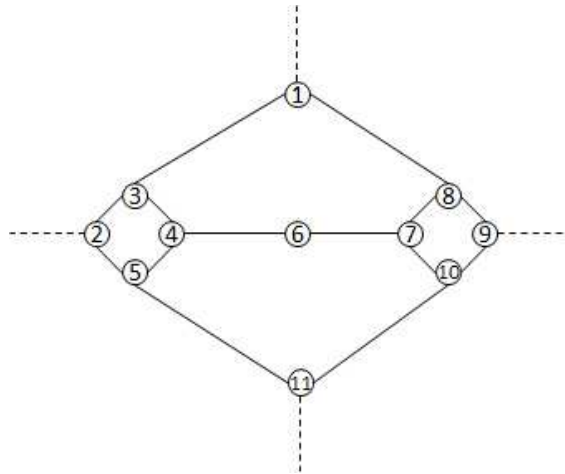}
\caption{A 4-gate, with the dashed lines representing the four external edges.\label{fig-4gate}}
\end{figure}

\begin{lemma}It is possible to enter the 4-gate via any of the external edges, and exit via any of the remaining external edges, visiting every vertex exactly once.\label{lem-4g}\end{lemma}

\begin{proof}It suffices to give paths between any two of the external edges. Since the 4-gate is undirected, the reverse path is obviously permitted, so we only need to consider unordered pairs. Also, due to symmetry,
the top and bottom edges are equivalent, as are the left and right edges. Then there are only three cases that need to be considered, which are displayed in Figure \ref{fig-4gate_paths}.

Top edge to left edge: The path is $1 - 3 - 4 - 6 - 7 - 8 - 9 - 10 - 11 - 5 - 2$.\\
Top edge to bottom edge: The path is $1 - 3 - 2 - 5 - 4 - 6 - 7 - 8 - 9 - 10 - 11$.\\
Left edge to right edge: The path is $2 - 3 - 1 - 8 - 7 - 6 - 4 - 5 - 11 - 10 - 9$. \end{proof}

\begin{figure}[h!]
\centering\hspace*{-0.5cm}\includegraphics[scale=0.4]{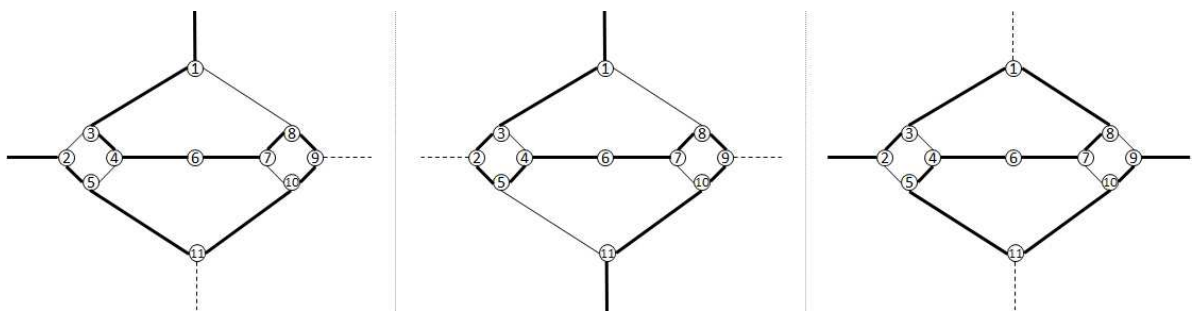}
\caption{The three paths through the 4-gate described in Lemma \ref{lem-4g}, displayed here as bold edges.\label{fig-4gate_paths}}
\end{figure}

In the following proof, and throughout the remainder of this manuscript, we make use of the concept of a {\em live edge}, being one that may still be used without creating a short cycle. In general, if a vertex $v$ has only two
live edges, and an adjacent vertex is visited, then vertex $v$ must be visited immediately afterwards in order for a Hamiltonian cycle to be formed.

\begin{proposition}Upon entering the 4-gate, every vertex must be traversed before exiting in a Hamiltonian cycle.\label{prop-4g}\end{proposition}

\begin{proof}Suppose that during the course of a Hamiltonian cycle, the 4-gate is entered, and then exited before all vertices are visited. Then the Hamiltonian cycle must later enter and exit the 4-gate again. Therefore, one such
path must enter or exit through the top edge. Since the graph is undirected, without loss of generality we may assume that the top edge is entered. Then, suppose the cycle travels from vertex 1 to 3. At this point,
the cycle may either continue to vertex 2 or 4.

If the cycle continues to vertex 4, it must then continue to the degree 2 vertex 6, and on to vertex 7. Then, since edge $(1,8)$ was not used, there are only two remaining live edges adjacent to vertex 8, so the cycle must continue
to vertex 8 and through to vertex 9. At this point it must exit the right edge (or else it will be impossible to re-enter and re-exit the 4-gate later). Clearly then, since edges $(7,10)$ and $(9,10)$ were not used, it is now impossible to
visit vertex 10 without getting stuck.

If, instead, the cycle continues to vertex 2, then using the same argument as above, it must exit via the left edge. Then, some time later, the cycle re-enters the 4-gate. Again, without loss of generality, suppose it enters via the right edge. Then using the same
argument as above, the cycle is forced to travel the path $9 - 8 - 7 - 6 - 4 - 5 - 11$. At this stage, the cycle must exit via the bottom edge, as it is the only remaining external edge. However, it is then impossible to visit
vertex 10. So we conclude that the initial choice of travelling from vertex 1 to 3 is flawed.

However, due to symmetry, travelling from vertex 1 to 8 will be similarly flawed. Therefore the initial assumption that the 4-gate was exited before all vertices was visited must be incorrect. \end{proof}

From Lemma \ref{lem-4g} and Proposition \ref{prop-4g} we see that the 4-gate functions the same as a degree 4 vertex. That is, once it is entered via one edge, any of the other three edges can be departed from, but only once the entire
4-gate has been traversed. Then, the procedure to convert a sub-quartic instance to a sub-cubic instance is as follows.

{\underline{Sub-quartic HCP to Sub-cubic HCP Conversion Procedure}}

\begin{itemize}\item Replace any degree 4 vertices with a 4-gate, with the four adjacent edges to the degree 4 vertex forming the four external edges to the 4-gate.\end{itemize}

Since the 4-gate subgraphs are sub-cubic, and all remaining vertices in the original instance are degree 3 or less, the resulting instance is now sub-cubic. It is clear that the conversion from sub-quartic HCP to sub-cubic HCP is
a linearly-growing conversion, where in the worst case there are 11 times as many vertices, and 4.5 times as many edges. Then, since there is a linearly-growing conversion from sub-cubic HCP to cubic HCP, we conclude that there is
a linearly-growing conversion from sub-quartic HCP to cubic HCP.

We further extend the above idea by introducing the {\em 5-gate}, displayed in Figure \ref{fig-5gate}.

\begin{figure}[h!]
\centering\hspace*{-0.5cm}\includegraphics[scale=0.75]{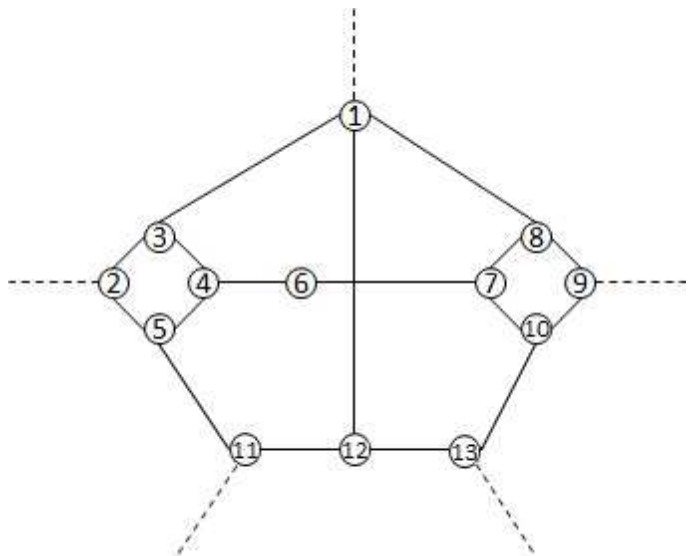}
\caption{A 5-gate, with the dashed lines representing the five external edges.\label{fig-5gate}}
\end{figure}

\begin{lemma}It is possible to enter the 5-gate via any of the external edges, and exit via any of the remaining external edges, visiting every vertex exactly once.\label{lem-5g}\end{lemma}

\begin{proof}Again, it suffices to give the paths between any pair of external edges. Due to symmetry we can eliminate some options, leaving only the following six paths, which are displayed in Figure \ref{fig-5gate_paths}.

Top edge to left edge: The path is $1 - 3 - 4 - 6 - 7 - 8 - 9 - 10 - 13 - 12 - 11 - 5 - 2$.\\
Top edge to bottom-left edge: The path is $1 - 3 - 2 - 5 - 4 - 6 - 7 - 8 - 9 - 10 - 13 - 12 - 11$.\\
Left edge to right edge: The path is $2 - 3 - 1 - 8 - 7 - 6 - 4 - 5 - 11 - 12 - 13 - 10 - 9$.\\
Left edge to bottom-left edge: The path is $2 - 3 - 1 - 12 - 13 - 10 - 9 - 8 - 7 - 6 - 4 - 5 - 11$.\\
Left edge to bottom-right edge: The path is $2 - 3 - 1 - 12 - 11 - 5 - 4 - 6 - 7 - 8 - 9 - 10 - 13$.\\
Bottom-left edge to bottom-right edge: The path is $11 - 5 - 2 - 3 - 4 - 6 - 7 - 10 - 9 - 8 - 1 - 12 - 13$. \end{proof}

\begin{figure}[h!]
\centering\hspace*{-0.5cm}\includegraphics[scale=0.4]{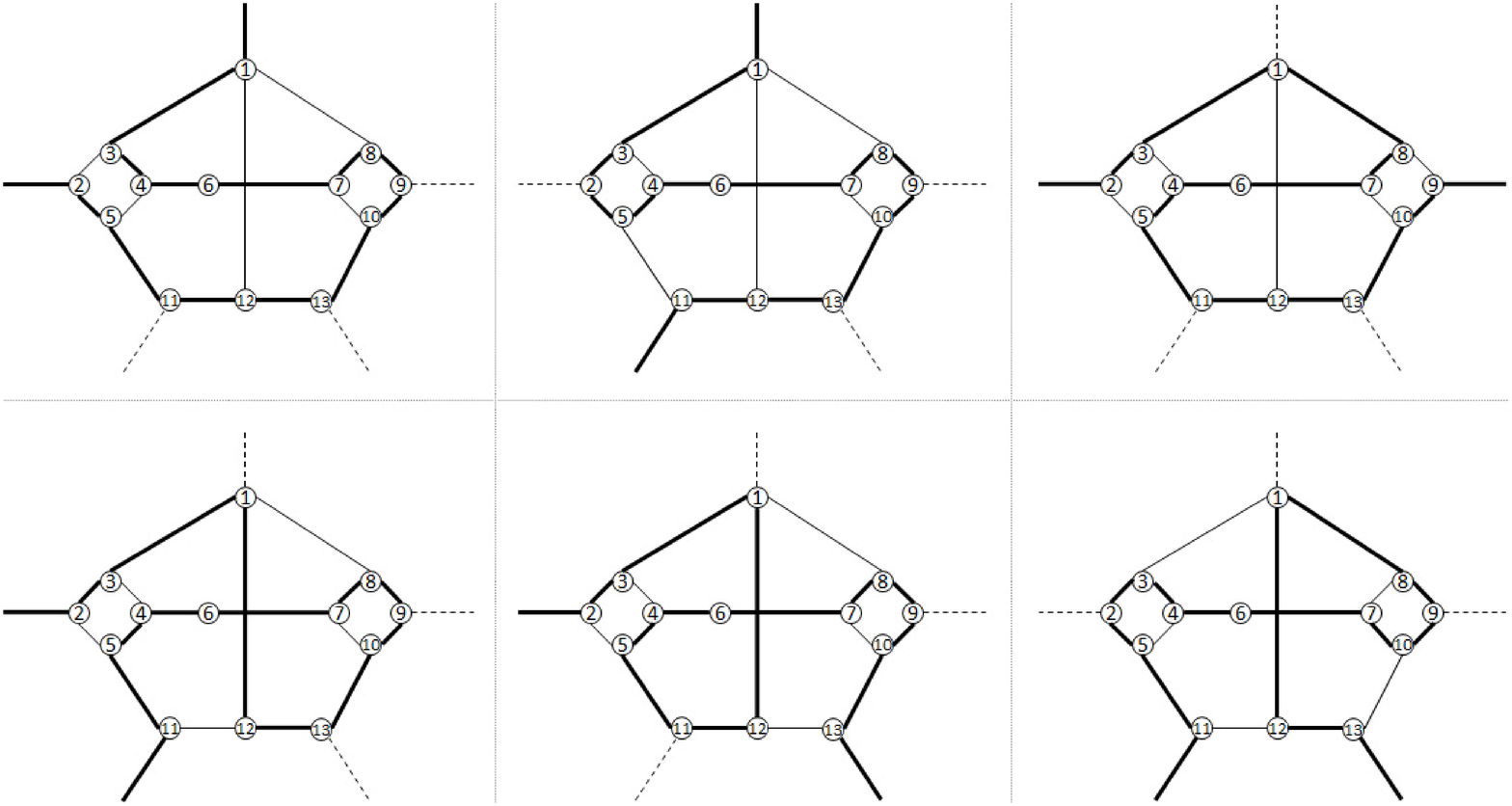}
\caption{The six paths through the 5-gate described in Lemma \ref{lem-5g}, displayed here as bold edges.\label{fig-5gate_paths}}
\end{figure}

\begin{proposition}Upon entering the 5-gate, every vertex must be traversed before exiting in a Hamiltonian cycle.\label{prop-5g}\end{proposition}

\begin{proof}Suppose that, during the course of a Hamiltonian cycle, the 5-gate is entered and then exited before all vertices are visited. Then the Hamiltonian cycle must later enter and exit the 5-gate again. Therefore all but
one external edge is traversed at some point. Suppose that the top external edge is one of the external edges used. Without loss of generality, assume that the top external edge is entered. Then, suppose that the cycle continues
from vertex 1 to 3. The possible situations that arise from this choice will be displayed in Figure \ref{fig-5gate_proof1}. Note that, due to symmetry, if this choice proves impossible, then it will also be impossible to travel to
vertex 8 as well.

At this point the path is $1 - 3$. Then, suppose the cycle continues to vertex 4. Using similar arguments as in the proof of Proposition \ref{prop-4g}, we see that the cycle must then continue through vertices $4 - 6 - 7 - 8 - 9$,
since vertices 4 and 8 both have only two live remaining edges. However, now vertex 10 has only two remaining live edges, and so the Hamiltonian cycle continues through vertices $9 - 10 - 13$. At this point the bottom-right external edge must be used, or else there will not be enough remaining external edges to re-enter and re-exit later. However, it is then impossible to visit vertex 12 without getting stuck. This situation is displayed in the first part of Figure \ref{fig-5gate_proof1}. Therefore we backtrack and choose to travel from vertex 3 to 2, instead of from vertex 3 to 4.

At this point the path is $1 - 3 - 2$. The cycle may now either exit the 5-gate here, or continue. Suppose that it exits here. Then consider the possibility that the right external edge is later used. Without loss of
generality, assume it is re-entered. Using the same arguments as above, we must travel along the path $9 - 8 - 7 - 6 - 4$, since vertices 8 and 6 have only two live edges at this point. However, it is then impossible to visit
vertex 10 without getting stuck. This situation is displayed in the second part of Figure \ref{fig-5gate_proof1}. So therefore, the right external edge cannot be re-entered or re-exited, and so the bottom-left and bottom-right external edges must both be used. Without loss of generality, assume the cycle re-enters
via the bottom-left external edge. Then, since vertex 12 has only two remaining live edges, it must be visited, and then the path continues to vertex 13 where it immediately departs. Then there are remaining vertices in the 5-gate
that can not be visited without getting stuck, as there is only one remaining external edge. Therefore this choice is impossible. This situation is displayed in the third part of Figure \ref{fig-5gate_proof1}. Then, we see that the
cycle must not exit via the left external edge, and must instead travel from vertex 2 to vertex 5.

At this point the path is $1 - 3 - 2 - 5$. Then there are only two live edges adjacent to vertices 4, 6 and 8, and so we are forced to continue through $4 - 6 - 7 - 8 - 9$. Then we must exit, to ensure there
are still two live external edges to re-enter and re-exit, but this means it is impossible to visit vertex 10 without getting stuck. This situation is displayed in the fourth part of Figure \ref{fig-5gate_proof1}.

\begin{figure}[h!]
\centering\hspace*{-0.5cm}\includegraphics[scale=0.6]{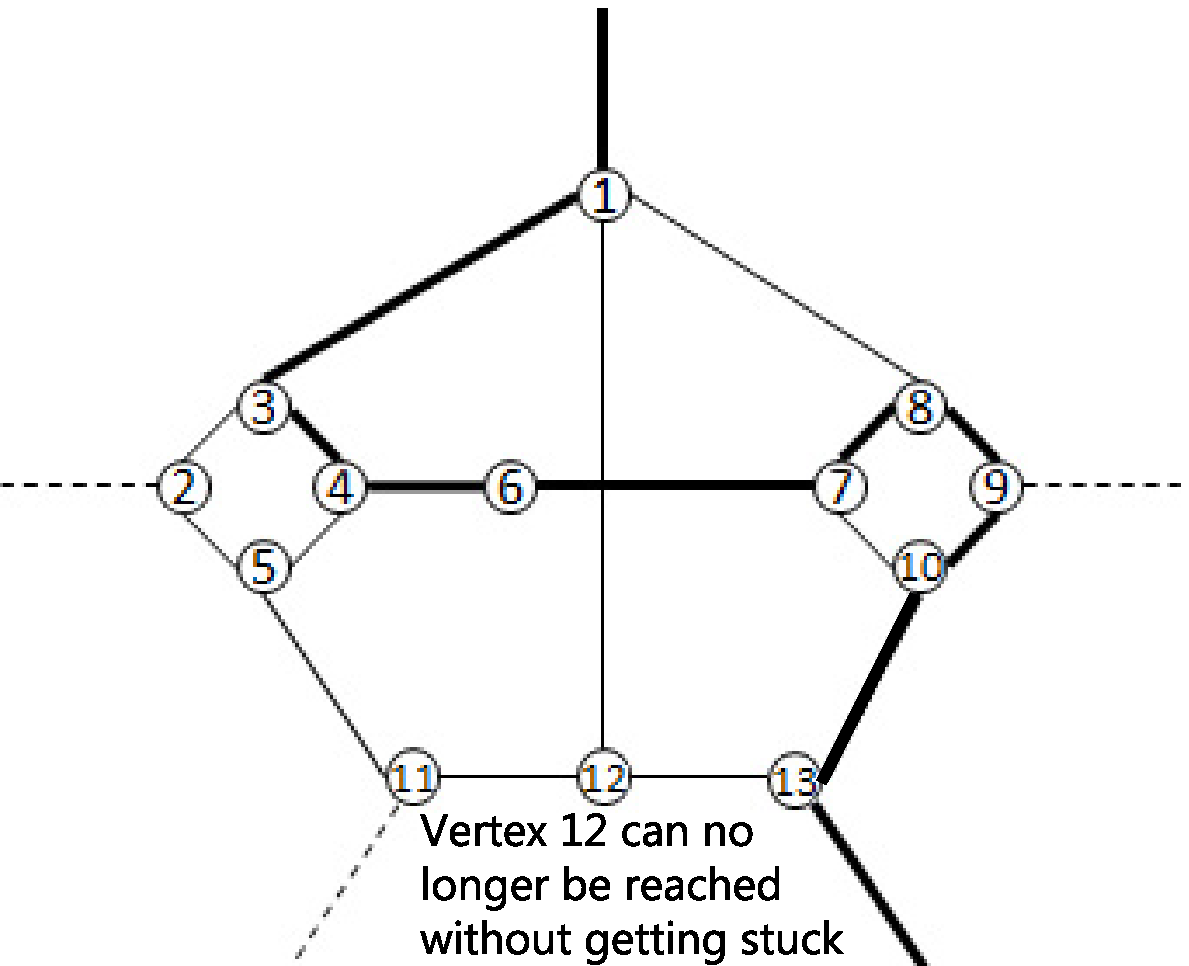} \;\;\;\;\; \includegraphics[scale=0.6]{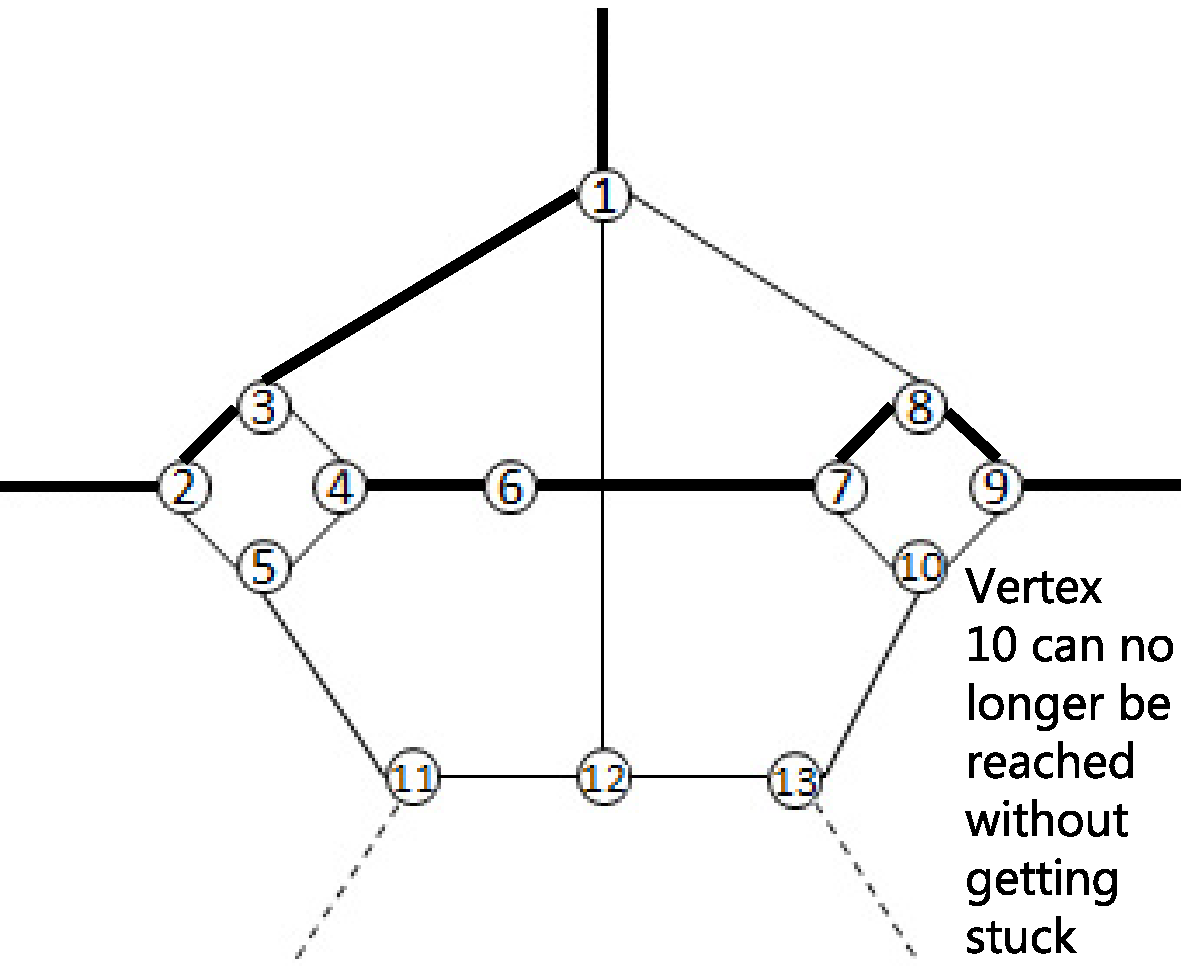}\\
\vspace*{0.67cm}
\hspace*{-0.5cm}\includegraphics[scale=0.6]{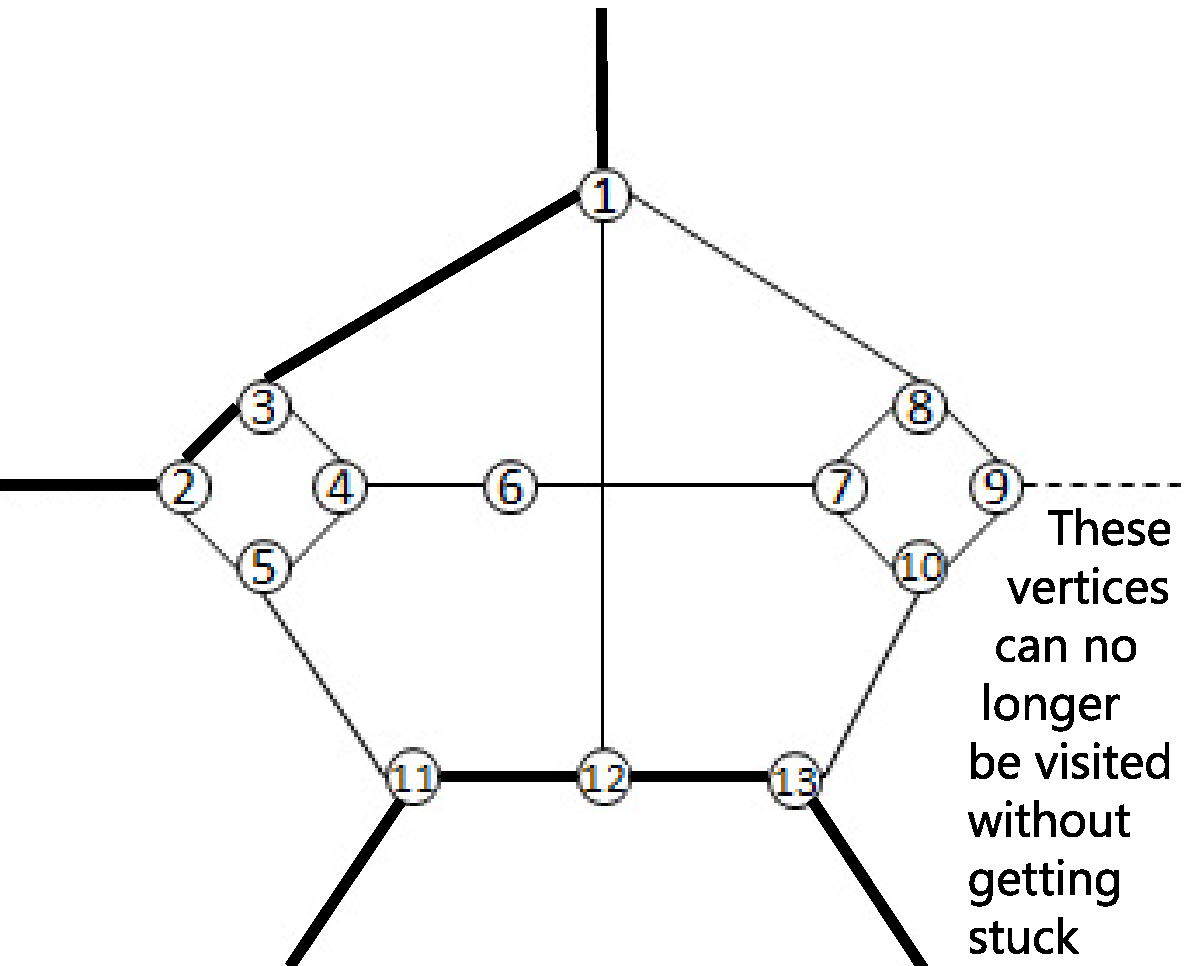} \;\;\;\;\; \includegraphics[scale=0.6]{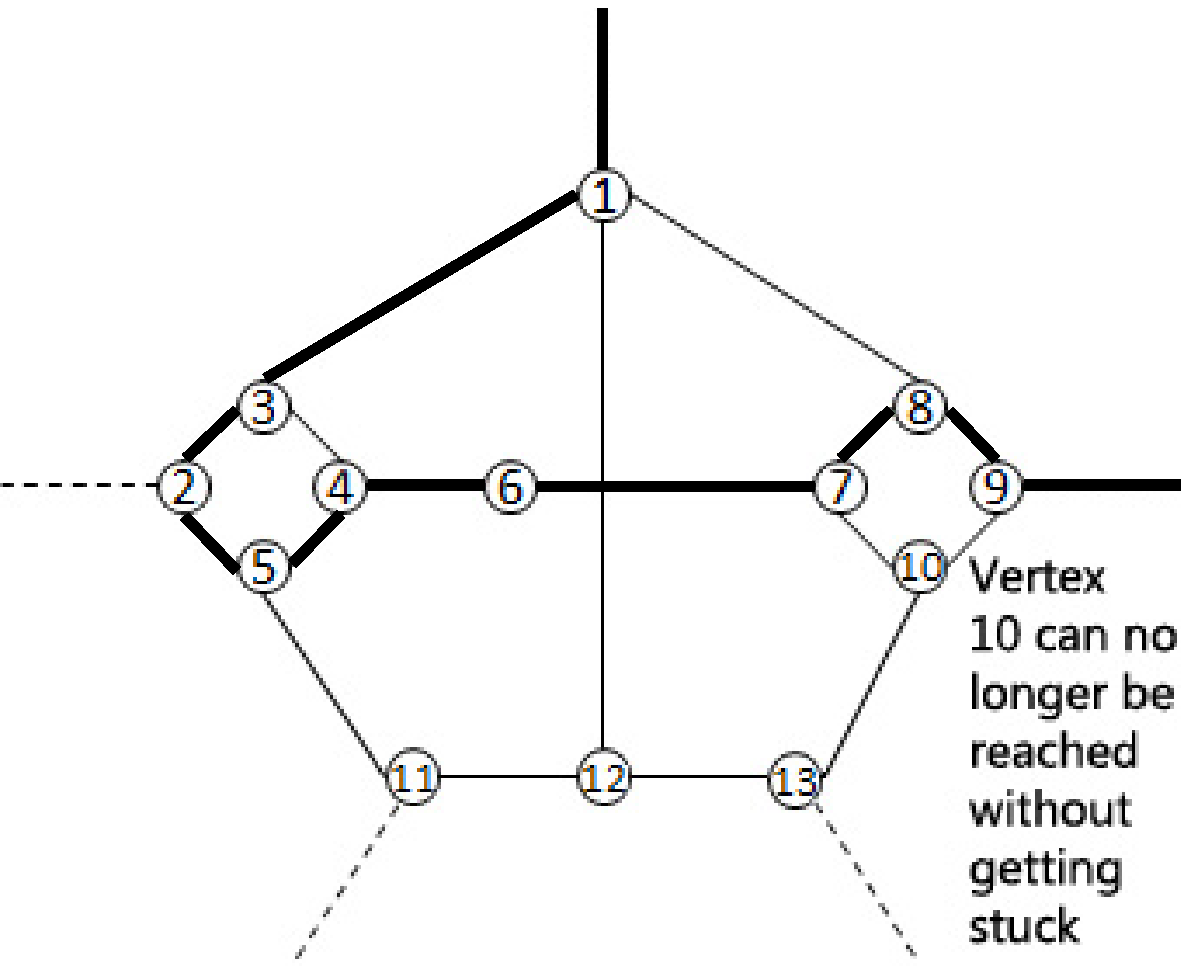}
\caption{The four possible situations that arise from the top external edge being entered, and edge $(1,3)$ being chosen. The paths taken are indicated by bold edges.\label{fig-5gate_proof1}}
\end{figure}

Therefore all possible situations that arise from entering via the top external edge and choosing the path beginning with $1 - 3$ are flawed, so we must backtrack all the way to the start. The remaining situations will be displayed
in Figure \ref{fig-5gate_proof2}. Recall that the top external edge was entered. The only remaining option to check is if we travel from vertex 1 to vertex 12. Then it is clear that vertices 3, 6 and 8 have only two live edges
remaining. That means at some point there will be a path going from $2 - 3 - 4 - 6 - 7 - 8 - 9$. At this point, we see that vertices 5 and 10 now also have only two live edges remaining, and so must be visited on this path, but this
means that neither the right nor left external edges are used. Then it is impossible to re-enter and re-exit later, so this choice is flawed as well. This situation is displayed in the first part of Figure \ref{fig-5gate_proof2}.

Having checked all possible options after entering via the top external edge, we conclude that it may not be used to enter the 5-gate. Without loss of generality, this also means it may not be exited. Consider the remaining
possibility, where the top external edge is not used as an entering or exiting external edge. Suppose that edges $(1,3)$ and $(1,8)$ are used in a path through the 5-gate. It is clear that vertex 12 has only two remaining live edges,
and since all external edges other than the top external edge must be used, one path through the 5-gate must be $11 - 12 - 13$. Then consider the path containing vertex 8. If the path continues to vertex 7, there are two adjacent
vertices, 6 and 10, each with two remaining live edges, so it is impossible to complete a Hamiltonian cycle. This situation is displayed in the second part of Figure \ref{fig-5gate_proof2}. Alternatively, if the path goes from vertex
8 to 9, it must then depart via the right external edge, but it is then impossible to reach vertex 10 without getting stuck. This situation is displayed in the third part of Figure \ref{fig-5gate_proof2}. Therefore this situation is
impossible.

Due to symmetry, we need only consider one remaining possibility. Suppose edges $(1,3)$ and $(1,12)$ are used in a path through the 5-gate. Then, the right external edge must be entered or exited at some stage.
Without loss of generality, suppose it is entered. Since vertices 8 and 6 have only two remaining live edges at this point, the path continues through $9 - 8 - 7 - 6 - 4$. However, it is then impossible to visit vertex 10 without
getting stuck, and this situation is impossible as well. This situation is displayed in the fourth part of Figure \ref{fig-5gate_proof2}.

\begin{figure}[h!]
\centering\hspace*{-0.5cm}\includegraphics[scale=0.6]{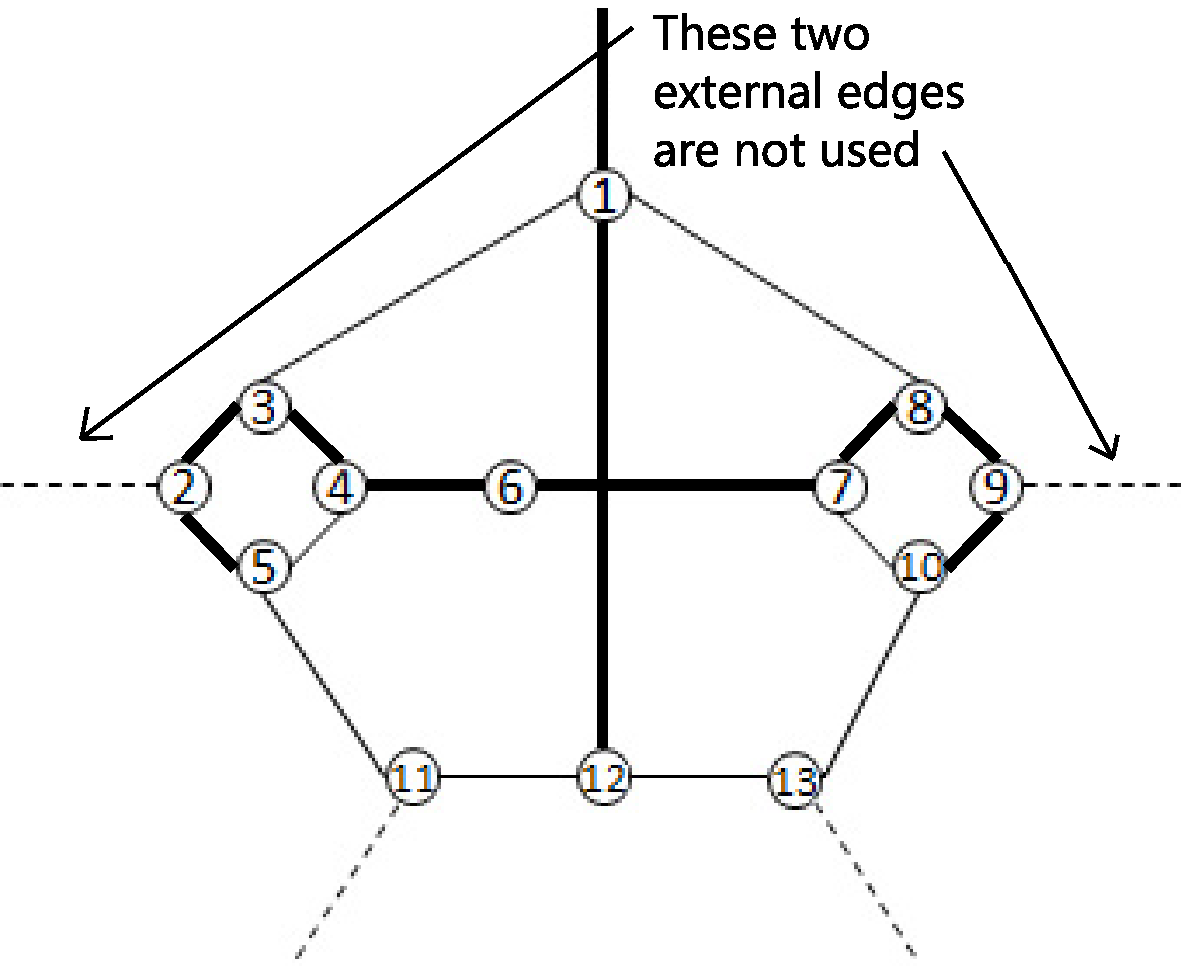} \;\;\;\;\; \includegraphics[scale=0.6]{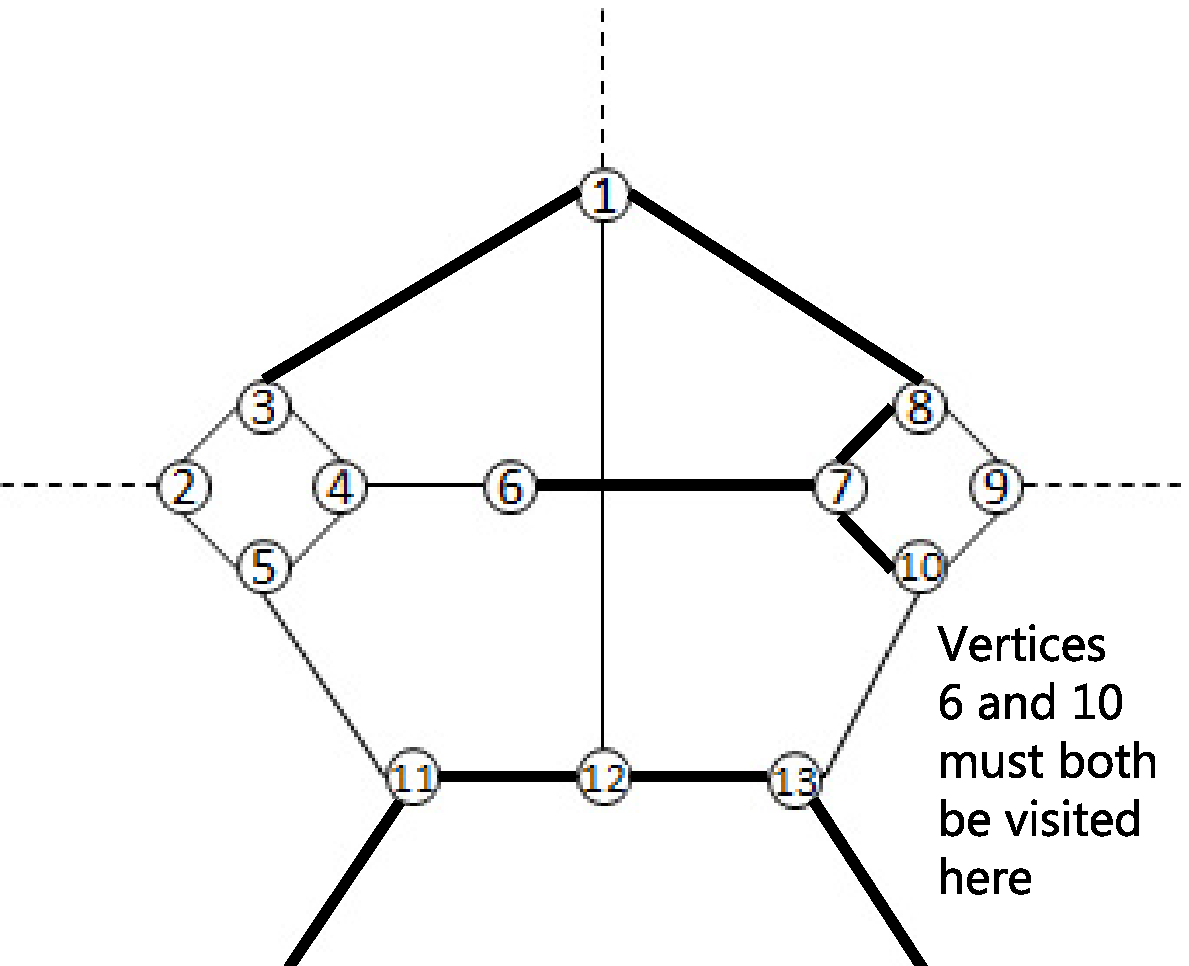}\\
\vspace*{0.67cm}
\hspace*{-0.5cm}\includegraphics[scale=0.6]{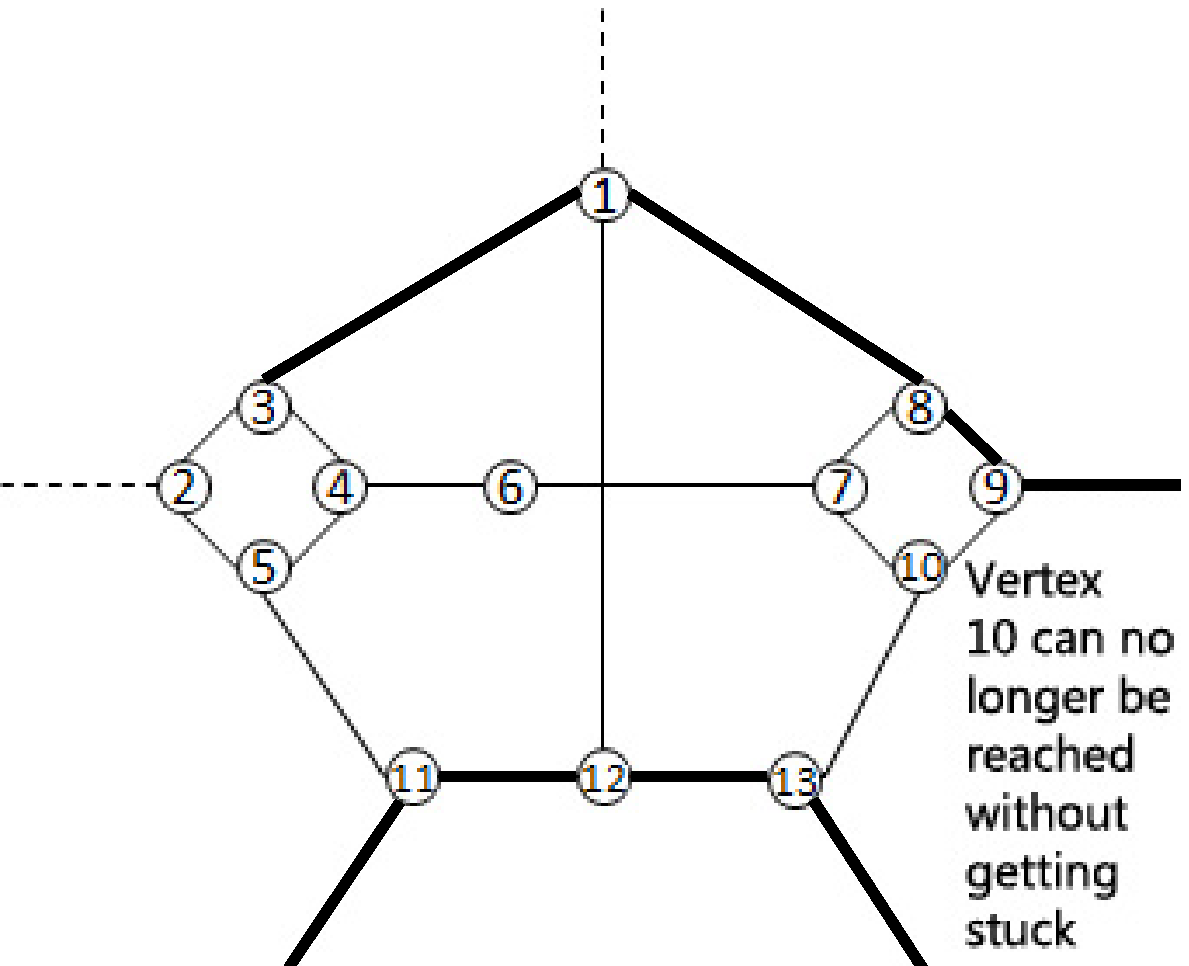} \;\;\;\;\; \includegraphics[scale=0.6]{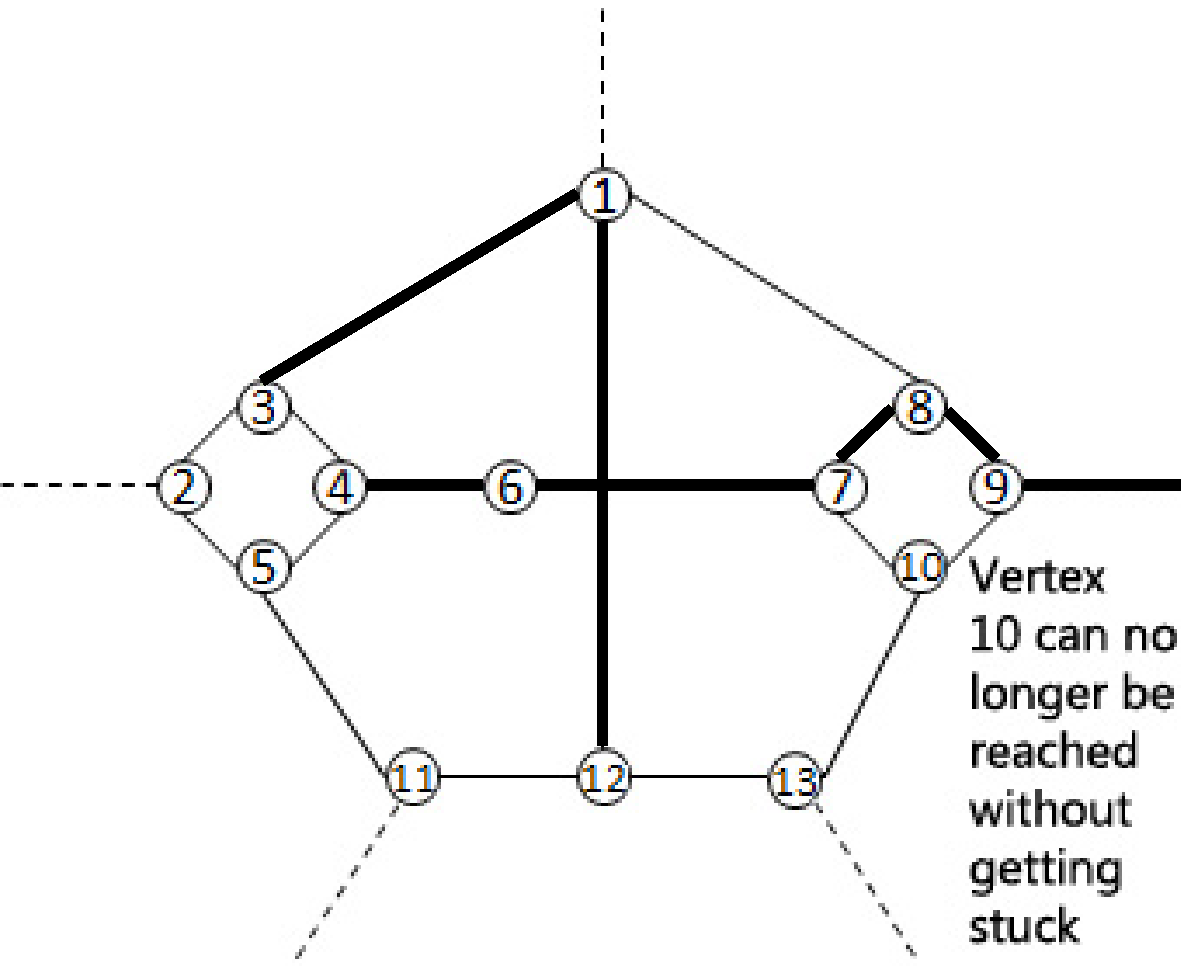}
\caption{The four possible situations that remain if we exclude the situation where the top external edge is entered and edge $(1,3)$ is chosen. The paths taken are indicated by bold edges.\label{fig-5gate_proof2}}
\end{figure}

Since all possibilities have been exhausted, and we must conclude that the initial assumption that the Hamiltonian cycle enters and exits the 5-gate without visiting all vertices, is false. \end{proof}

Note that the 5-gate is a sub-quartic subgraph. Then consider an undirected graph with maximum degree 5. We refer to HCP restricted to such instances as {\em sub-quintic HCP}. Then we can convert from sub-quintic HCP to
sub-quartic HCP simply:

{\underline{Sub-quintic HCP to Sub-quartic HCP Conversion Procedure}}

\begin{itemize}\item Replace any degree 5 vertices with a 5-gate, with the five adjacent edges to the degree 5 vertex forming the five external edges to the 5-gate.\end{itemize}

Since the 5-gate subgraph is sub-quartic, and all remaining vertices in the original instance are degree 4 or less, the resulting instance is now sub-quartic. It is clear that the conversion from sub-quintic HCP to sub-quartic HCP
is a linearly-growing conversion, where in the worst case there are 13 times as many vertices, and 4.4 times as many edges. Then, since there is a linearly-growing conversion from sub-quartic HCP to cubic HCP, we conclude that there
is a linearly-growing conversion from sub-quintic HCP to cubic HCP.

The concept of 4-gates and 5-gates is made general in the following section.

\section{Converting General HCP to Cubic HCP}

Consider the construction displayed in Figure \ref{fig-sgate}, which we call an $s$-gate for all integer $s \geq 4$.

\begin{figure}[h!]
\centering\hspace*{-0.5cm}\includegraphics[scale=0.625]{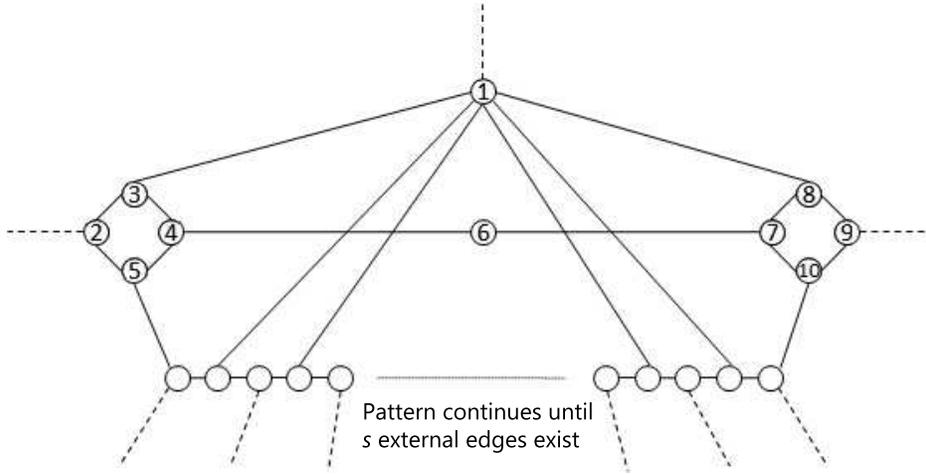}
\caption{An $s$-gate, with the dashed edges representing the $s$ external edges.\label{fig-sgate}}
\end{figure}

The bottom row is made up of $2s-7$ vertices with a path through them. Starting from the first vertex on the bottom row, every second vertex is adjacent to an external edge. The other vertices are adjacent to vertex 1 (the top
vertex).

\begin{lemma}It is possible to enter the $s$-gate via any of the external edges, and exit via any of the remaining external edges.\label{lem-sg}\end{lemma}

\begin{proof}Using equivalent paths to those in the 5-gate case, it is easy to see that it is still possible to find a path between any two of the top, left, right, bottom-left and bottom-right external edges. Then we only need to
consider cases in which one of the bottom external edges is used, and due to symmetry we can ignore the case of using the right or bottom-right external edges. We display the four situations in Figure \ref{fig-sgate_paths}.

Top edge to any bottom edge: First travel from vertex 1 down to whichever vertex is directly to the right of the exit vertex. Then travel right along the bottom row, up to 10, along the path
through $10 - 9 - 8 - 7 - 6 - 4 - 3 - 2 - 5$, down to the bottom left, and right along the bottom row until the exit vertex is reached.\\
Left edge to any bottom edge: First travel along the path through $2 - 3 - 1$, and down to whichever vertex is directly to the right of the exit vertex. Then travel right along the bottom row, up to 10, along the path
through $10 - 9 - 8 - 7 - 6 - 4 - 5$, down to the bottom left, and right along the bottom row until the exit vertex is reached.\\
Bottom-left edge to any bottom edge: First travel right along the bottom row until the vertex directly to the left of the exit vertex is reached. Then travel up to vertex 1, and through
path $1 - 3 - 2 - 5 - 4 - 6 - 7 - 8 - 9 - 10$, down to the bottom right, and left along the bottom row until the exit vertex is reached.\\
Any bottom edge to any other bottom edge: Without loss of generality, assume we enter on the left-most of the two bottom edges. Travel right along the bottom row until the vertex directly to the left of the exit vertex is reached.
Travel up to vertex 1, and then back down to the vertex directly to the left of the entrance vertex. Then travel left along the bottom row, up to 5, and along the path through $5 - 2 - 3 - 4 - 6 - 7 - 8 - 9 - 10$. Then go down to
the bottom right, and travel left along the bottom row until the exit vertex is reached. \end{proof}

\begin{figure}[h!]
\centering\hspace*{-0.5cm}\includegraphics[scale=0.4]{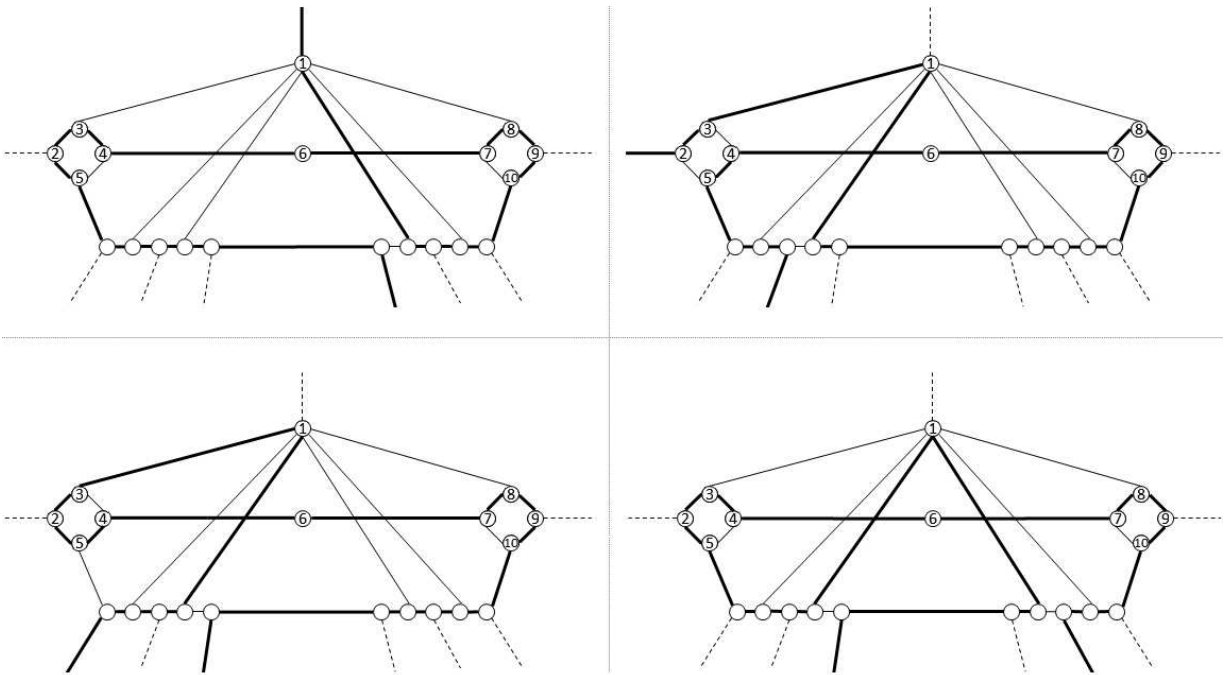}
\caption{The four paths through the $s$-gate involving one of the bottom external edges described in Lemma \ref{lem-sg}, displayed here as bold edges.\label{fig-sgate_paths}}
\end{figure}

We now pose two small results that will be used to shorten the proof of the upcoming Proposition \ref{prop-sg}. These results show that the $s$-gate is equivalent to the $5$-gate under special conditions.

\begin{lemma}If a path enters the $s$-gate, and at some point visits all vertices on the bottom row in succession, then it must visit all other vertices in the $s$-gate before exiting.\label{lem-sit1}\end{lemma}

\begin{proof}It is easy to see that this situation is identical to that of the 5-gate in which the path $11 - 12 - 13$ is chosen. Then the result follows from Proposition \ref{prop-5g}. \end{proof}

\begin{lemma}If a path enters the $s$-gate, and at some point visits one group of vertices on the bottom row in succession, and later the remaining vertices on the bottom row are visited in succession, then the path must
visit all other vertices in the $s$-gate before exiting.\label{lem-sit2}\end{lemma}

\begin{proof}Similarly to Lemma \ref{lem-sit1}, it is easy to see that this situation is identical to that of the 5-gate in which either the path $11 - 12 - 1$ is chosen and 13 is visited later, or the path $13 - 12 - 1$ is chosen
and 11 is visited later. Then the result follows from Proposition \ref{prop-5g}. \end{proof}

We now present the main result of this section. During the proof, we will consider multiple paths traveling through the $s$-gate, which each make up a small part of a larger Hamiltonian cycle, and show that in all cases, this is
impossible.

\begin{proposition}A Hamiltonian cycle, upon entering an $s$-gate, must traverse every vertex before exiting.\label{prop-sg}\end{proposition}

\begin{proof}Suppose that during the course of a Hamiltonian cycle, the $s$-gate is entered, and then exited before all vertices are visited. Then the $s$-gate must be re-entered later. So there are at least two disjoint paths
travelling through the $s$-gate. Since we have a Hamiltonian cycle, we know that any given vertex in the $s$-gate must lie on one of the paths. Consider the path containing vertex 1. We will now consider three possible cases,
relating to which edges adjacent to vertex 1 are used in this path.

Case 1: Edges $(1,3)$ and $(1,8)$ are used. Then all other bottom-row vertices adjacent to vertex 1 have only two live edges. Therefore once either the bottom-left or bottom-right vertex is visited (either on this path, or a
different path), the bottom vertices must all be visited in succession. However, by Lemma \ref{lem-sit1}, we know then that the path containing the bottom vertices must contain all vertices on the $s$-gate, which violates the initial
assumption. So this case cannot occur. This situation is displayed in the first part of Figure \ref{fig-sgate_proof1}.

Case 2: Exactly one edge going from vertex 1 to a bottom-row vertex is used. Call this edge $(1,a)$. Then all vertices adjacent to 1 in the bottom row, except for vertex $a$, have two live edges remaining. Clearly then the vertices
on one side of $a$ must be visited in succession. Without loss of generality, we assume the vertices to the right of $a$ are visited in succession. Then, at some point later (either on this path, or a different path), the
bottom-left vertex is visited. Then either directly before, or directly after this point, all of the remaining bottom vertices must be visited in succession. However, by Lemma \ref{lem-sit2}, we know then that this path contains all
vertices on the $s$-gate, which violates the initial assumption. Therefore this case cannot occur either. This situation is displayed in the second part of Figure \ref{fig-sgate_proof1}.

\begin{figure}[h!]
\centering\hspace*{-0.5cm}\includegraphics[scale=0.4]{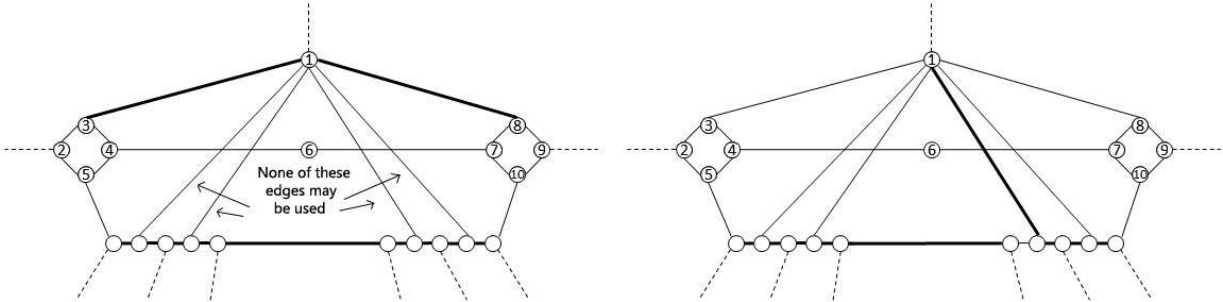}
\caption{The first two cases in Proposition \ref{prop-sg}, where the paths are indicated by bold edges. These two situations are both equivalent to the 5-gate by Lemmata \ref{lem-sit1} -- \ref{lem-sit2}.\label{fig-sgate_proof1}}
\end{figure}

Case 3: Exactly two edges going from vertex 1 to a bottom-row vertex are used. Call these edges $(1,a)$ and $(1,b)$, and without loss of generality suppose that vertex $b$ is to the right of $a$. Then vertices 3, 6 and 8 all have
only two live edges remaining. It is clear then that one of the paths through the $s$-gate contains the fragment $2 - 3 - 4 - 6 - 7 - 8 - 9$. Also, if this path enters or exits via the left external edge, then it will be impossible
to visit vertex 5. Likewise, if the path enters or exits via the right external edge, it will be impossible to visit vertex 10. Since neither can occur, we see that this path contains a fragment originating at the bottom left vertex,
travelling through $5 - 2 - 3 - 4 - 6 - 7 - 8 - 9 - 10$ and down to the bottom right vertex. Also, due to vertices with only two live edges remaining, from the bottom-right vertex the path continues left along the bottom row until
it reaches the vertex to the right of $b$. Similarly, before reaching the bottom-left vertex, the path must have travelled from left along the bottom row, originating at the vertex to the left of $a$. This situation is shown in
Figure \ref{fig-sgate_proof2}, with the edges forced to be on the path displayed in bold.

\begin{figure}[h!]
\centering\hspace*{-0.5cm}\includegraphics[scale=0.62]{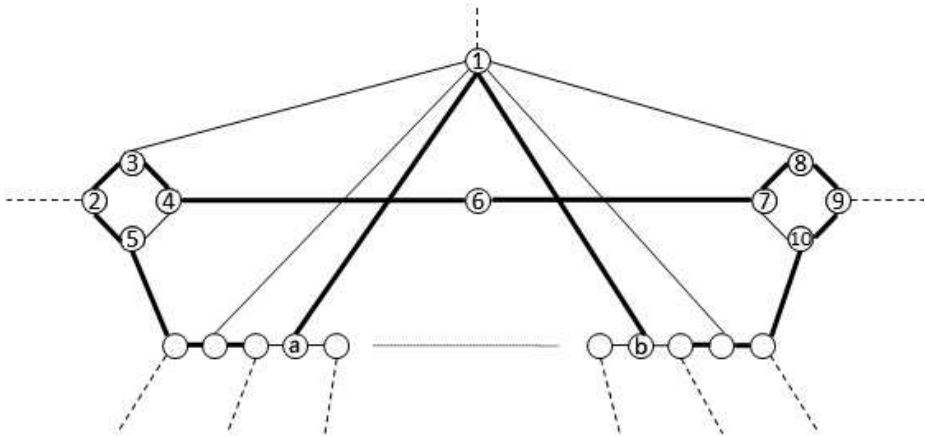}
\caption{Case 3 in Proposition \ref{prop-sg}, where the paths are indicated by bold edges.\label{fig-sgate_proof2}}
\end{figure}

Consider the path containing vertices $a - 1 - b$. After leaving vertex $b$, the path must continue either left or right. We will consider these two final cases.

Case 3.1: Suppose the path goes to the left after vertex $b$. Then due to the vertices with only two live edges remaining, the path continues until the vertex to the right of $a$, at which point the path must exit the $s$-gate.
However, then the path $a - 1 - b$ must be joined to the other, forced, path. This combined path visits every vertex in the $s$-gate, violating the initial assumption, and
therefore this case cannot occur. This situation is displayed in the first part of Figure \ref{fig-sgate_proof3}.

Case 3.2: Suppose the path goes to the right after vertex $b$. Then it joins the other, forced, path.. Now, consider the vertex to the left of $a$. It cannot be joined to $a$ on the path, or else a cycle is  closed off, which violates the assumption that a Hamiltonian cycle is formed. So therefore, before visiting vertex $a$, the path must have visited the vertex to the right of vertex $a$, but due to vertices with only two live remaining edges, this means all vertices between $a$ and $b$ are on this path as well. Then this path visits every vertex in the $s$-gate, violating the initial assumption, and therefore this case cannot occur. This situation is displayed in the second part of Figure \ref{fig-sgate_proof3}.

\begin{figure}[h!]
\centering\hspace*{-0.5cm}\includegraphics[scale=0.4]{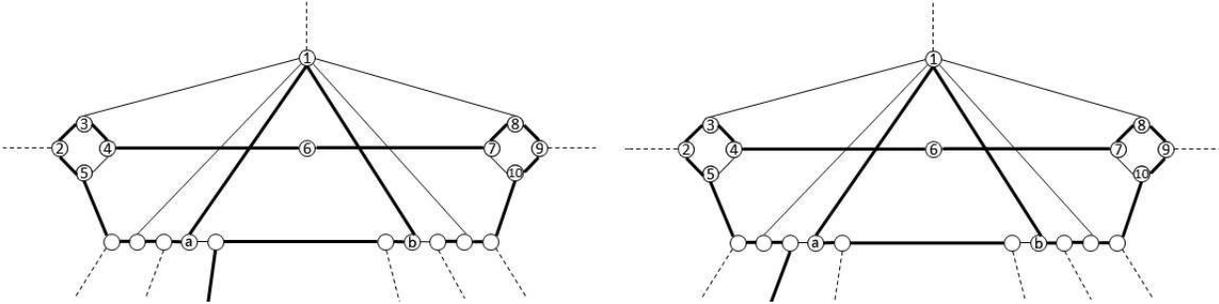}
\caption{Case 3.1 and 3.2 in Proposition \ref{prop-sg}, where the paths are indicated by bold edges. In each case, a single path contains every vertex in the $s$-gate, violating the initial assumption.\label{fig-sgate_proof3}}
\end{figure}

Since no potential cases may occur, the initial assumption that a Hamiltonian cycle may leave the $s$-gate without first visiting all vertices is flawed, and the proof is complete. \end{proof}

Note that the $s$-gate is sub-cubic except for vertex 1, which has degree $s-1$. Then, inductively, we may substitute in smaller and smaller gates until a sub-cubic subgraph is reached. Since an $s$-gate has $2s+3$ vertices, after
substituting in the $i$-gate we obtain $2i+2$ additional vertices of maximum degree 3, plus one more vertex of degree $i-1$. Continuing until $i = 4$, we see that after the induction is complete, there are
\begin{eqnarray*}1 + \sum_{i=4}^s \left(2 + 2i\right) & = & s^2 + 3s - 17 \mbox{ vertices.}\end{eqnarray*}
Similarly, in an $s$-gate there are $3s+2$ edges (not including the $s$ external edges). Each time we substitute in a smaller gate, no edges are destroyed, and so at stage $i$, there are $3i+2$ new edges. Continuing
until $i = 4$, we see that after the induction is complete, there are
\begin{eqnarray*}s + \sum_{i=4}^s 3i + 2 & = & \frac{3s^2 + 9s - 48}{2} \mbox{ edges, including external edges.}\end{eqnarray*}

We may now propose a very simple conversion algorithm from undirected HCP to sub-cubic HCP.

{\underline{Undirected HCP to Sub-cubic HCP Conversion Procedure}}

\begin{itemize}\item If any vertices with degree $s \geq 4$ exist, replace them with an $s$-gate.
\item Repeat the above process until all vertices have maximum degree 3.\end{itemize}

Note that if we then proceed to convert the graph from sub-cubic to cubic, $s-3$ degree 2 vertices are replaced by diamond subgraphs, and so the resultant instance will have $s^2 + 6s - 26$ vertices and $\frac{3s^2 + 19s - 78}{2}$ edges, including external edges. It should also be noted that there are multiple ways to traverse the $s$-gate, so the Hamiltonian cycles in the original graph have a 1-many relationship with the Hamiltonian cycles in the graph produced by the conversion algorithm.

\begin{lemma}The conversion process described above is a quadratically-growing conversion.\label{lem-quad}\end{lemma}

\begin{proof}Consider an undirected HCP instance containing $N$ vertices. Denote by $d_i$ the degree of vertex $i$. Then, if we define $k := \sum\limits_{i=1}^N d_i$, it is clear that $k$ is equal to twice the number of edges in the
undirected HCP instance. Since the number of edges will be greater than the number of vertices (or else the instance will be trivially non-Hamiltonian), it is clear that $k$ is of the same order as the input size of
the undirected HCP instance. Then, if we perform the above conversion on all vertices in the undirected HCP instance, it is clear then that in the worst case, there will be $O\left(\sum\limits_{i=1}^N d_i^2\right)$ vertices and edges
in the sub-cubic HCP instance. It can also be easily seen that $k^2 = O\left(\sum\limits_{i=1}^N d_i^2\right)$. Therefore, in the worst case, this conversion results in no more than $O(k^2)$ vertices and edges in the sub-cubic HCP
instance, and is therefore a quadratically-growing conversion. Finally, after the linearly-growing conversion from sub-cubic HCP to cubic HCP, we obtain a quadratically-growing conversion from undirected HCP to cubic HCP. \end{proof}

We now consider general HCP. The only distinction between general HCP and undirected HCP is that directed edges are permitted in instances of the former problem. In such a case, it makes sense to think not of the
degree for each vertex, but rather the in-degree $s$ and out-degree $r$ for each vertex. Of course, when considering such a problem we are permitted to simply convert it first to an undirected instance (a linearly-growing conversion), and then proceed using the above conversion. However, in doing so, we raise the maximum degree by 1, which may be inefficient, and triple the number of vertices before beginning the quadratically-growing conversion. An alternative approach is to simply use the above conversion, but to orient the external edges as necessary. So if a vertex has $s$ incoming edges and $r$ outgoing edges, an gate of size $s + r$ may be constructed whereby $s$ of the external edges are incoming edges and $r$ of the external edges are outgoing edges. We refer to such a gate, with directed external edges, as an $(s,r)$-gate.

To see that simply orienting the appropriate external edges will work, it is sufficient to recognise that directed edges behave exactly like undirected edges in the context of HCP, except that they are \lq\lq blocked" from one end. By orienting the external edges, a cycle will be similarly blocked from entering or exiting via an inappropriate external edge.

Of course, in practice, there might be cases of an incoming edge coming from, and an outgoing edges going to, the same neighbouring vertex, in which case the two edges may be treated as a single undirected edge. In general, if a
vertex $v$ has $u$ undirected edges, $s$ other incoming edges and $r$ other outgoing edges, then a gate of size $u + s + r$, with $u$ undirected external edges and $s+r$ directed external edges, may be used. However, for the remainder of this manuscript, we will use $(s,r)$-gates for the sake of simplicity of arguments.

Following from the preceding discussion and Lemma \ref{lem-quad}, we conclude this section with an important corollary.

\begin{corollary}If we restrict general HCP to instances in which the maximum in-degree and maximum out-degree are both bounded above by a fixed constant, then the conversion to cubic HCP described above is
linearly-growing.\label{col-linear_conversion}\end{corollary}

\begin{proof}Suppose the general HCP instance contains $N$ vertices, and that the maximum in-degree and maximum out-degree of the graph are both bounded above by a fixed constant $d$. Then, from Lemma \ref{lem-quad}, we see that
in the sub-cubic HCP instance there will be $O\left(\sum\limits_{i=1}^N d^2\right)$ vertices and edges, and since $d = O(1)$, this simplies to $O(N)$. \end{proof}

We use Corollary \ref{col-linear_conversion} in the following section, in which we describe a linearly-growing conversion from general HCP to general HCP of bounded in-degree and out-degree.

\section{A Linearly-growing Conversion from General HCP to Cubic HCP}

The final approach to the conversion of general HCP to cubic HCP is quite simple to implement, but proving that it is linearly-growing requires all of the above results. During the conversion we will make use of three special
subgraphs, which we call a split, an in-split and an out-split.

A split contains two vertices, and the directed edge $(1,2)$. There are external edges adjacent to vertices 1 and 2, such that the external edges adjacent to vertex 1 are all incoming edges, and the external edges adjacent to vertex 2 are all outgoing edges. An in-split contains three vertices, and the directed edges $(1,2)$, $(1,3)$, $(2,1)$ and $(2,3)$, and has external edges adjacent to each vertex. The external edges adjacent to vertices 1 and 2 are all incoming edges, and the external edges adjacent to vertex 3 are all outgoing edges. An out-split contains three vertices, and the directed edges $(1,2)$, $(1,3)$, $(2,3)$ and $(3,2)$, and has external edges adjacent to each vertex. The external edges adjacent to vertex 1 are all incoming edges, while the external edges adjacent to vertices 2 and 3 are all outgoing edges. These three subgraphs are displayed in Figure \ref{fig-split_graphs}.

\begin{figure}[h!]
\centering\hspace*{-0.5cm}\includegraphics[scale=0.67]{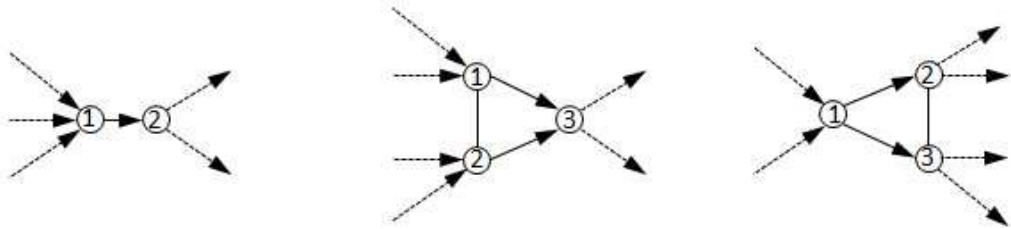}
\caption{A split, an in-split, and an out-split respectively. The arrow-less edge in each of the in-split and the out-split represents two directed edges between the same two vertices, which functions as an undirected edge. The dashed edges represent directed external edges.\label{fig-split_graphs}}
\end{figure}

Consider a graph containing any of the above three subgraphs. It is trivial to check that any Hamiltonian cycle, upon entering the subgraph, must traverse every vertex before exiting. It is also trivial to check that it is possible
to travel from any incoming external edge to any outgoing external edge. With this understanding in mind, we can easily see that replacing vertices by any of the above subgraphs does not alter Hamiltonicity. It is also interesting to note that the number of Hamiltonian cycles does not grow when a vertex is replaced by one of the above subgraph.

We now consider a conversion using the above three subgraphs. Suppose that a general HCP instance is given. It will be assumed that every edge in the graph is directed, and so any undirected edges should be thought of as two individual directed edges. The objective will be to replace all vertices of large in-degree or large out-degree (or both) with subgraphs, such that the final graph has maximum in-degree and maximum out-degree below a fixed constant, say $d$. We will choose $d = 4$, but as will become clear later, $d$ may be chosen for any integer value of 4 or above.

Consider a particular vertex $v$ in the graph, with in-degree $s$ and out-degree $r$, where $\max(s,r) > d$. We may replace this vertex with a subgraph by a procedure called the Splitting Procedure, which we outline below. Note that throughout the Splitting Procedure we refer to replacing vertices with the subgraphs described above. This should be done in such a way that the incoming edges adjacent to the replaced vertex form the incoming external edges in the subgraph, and likewise for the outgoing edges. For the in-split, there are two sets of incoming external edges, and so the incoming edges adjacent to the replaced vertex should be shared equally (or different by one, if there is an odd number) between these two sets. Likewise, for the out-split, there should be an equal share of the outgoing edges in each of the two sets of outgoing external edges. An example of such a replacement, using an in-split, is displayed in Figure \ref{fig-in_split_replaced}.

\begin{figure}[h!]
\centering\hspace*{-0.5cm}\includegraphics[scale=0.67]{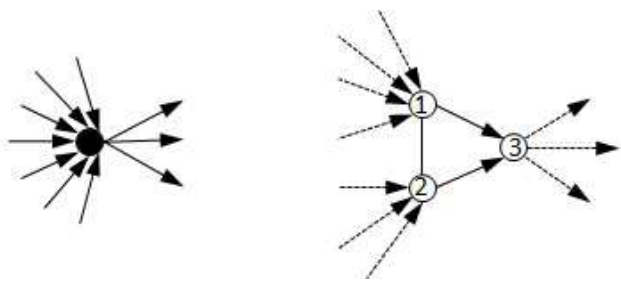}
\caption{A vertex with in-degree 7 and out-degree 3, and the corresponding in-split that replaces it. Note that the 7 incoming edges have been shared between vertices 1 and 2.\label{fig-in_split_replaced}}
\end{figure}

{\underline{Splitting Procedure}}

\begin{itemize}\item Replace vertex $v$ with a split.
\item While any vertex in the subgraph has in-degree greater than $d$, replace that vertex with an in-split.
\item While any vertex in the subgraph has out-degree greater than $d$, replace that vertex with an out-split.
\end{itemize}

\begin{lemma}The Splitting Procedure terminates in finite time for any $d \geq 4$.\label{lem-splitting_converges}\end{lemma}

\begin{proof}Consider a vertex with in degree $s$ and out-degree $r$. The first step of the Splitting Procedure occurs only once. After it has concluded, the first vertex in the split will have in-degree $s$ and out-degree 1, and the second vertex will have in-degree 1 and out-degree $r$. It is important to note here that neither of these vertices will have both in-degree and out-degree greater than $d$.

Next, the second step runs for as long as vertices with in-degree greater than $d$ exist. Initially there can only be one vertex (the first vertex) that satisfies this requirement. Then, each time a vertex $i$ with in-degree $s_i$ is replaced with an in-split, the maximum in-degree of the replacing subgraph is $\lceil \frac{s_i}{2} \rceil + 1$. It is easy to see that the resultant in-degree will shrink by an integer amount for any $s_i \geq 4$. Therefore step 2 will conclude in finite time for any $d \geq 4$.

The final step is equivalent to the second step, except the arguments involve the out-degree. Therefore, this step will also conclude for $d \geq 4$, and the algorithm will therefore terminate in finite time.\end{proof}

Since the Splitting Procedure terminates in finite time, and cannot alter the Hamiltonicity of the graph, the Splitting Procedure is guaranteed to convert any vertex of large in-degree or large out-degree (or both) to an equivalent subgraph of maximum in-degree and out-degree of $d$. Then if this is performed on all vertices in the graph, the resultant graph instance is equivalent to the original graph instance, but has in-degree and out-degree bounded above by $d$. The only remaining task is to determine how large the resultant graph will be after the Splitting Procedure is applied to all vertices. In the following proposition we consider one vertex $v$ in a graph, and show that the number of vertices in the replacing subgraph is a linear function of the in-degree and out-degree of $v$.

\begin{proposition}Consider a vertex $v$ with in-degree $s$ and out-degree $r$. After the Splitting Procedure is completed for a chosen $d \geq 4$, the resultant subgraph that replaces $v$ has $O(\max(s,r))$ vertices and $O(\max(s,r))$ edges.\label{prop-degree_d}\end{proposition}

\begin{proof}After the first step of the Splitting Procedure, the subgraph has 2 vertices. This includes a vertex with in-degree $s$ and out-degree 1, and a vertex with out-degree $r$ and in-degree 1. It is important to note here that there are no vertices in the subgraph that have both in-degree greater than $d$ and out-degree greater than $d$. Call this process "Stage 0".

In the second step of the Splitting Procedure, we may replace the first vertex with an in-split (if $s > d$), and in the third step we may replace the third vertex with an out-split (if $r > d$). Call this process "Stage 1". At this point the subgraph has at most six vertices. Two of them, one from the in-split, and one from the out-split, will require no further work, as they will all have in-degree and out-degree less than $d$. So at this point there are at most four vertices that require additional work. The in-splits will contain two vertices with potentially large in-degree, and the out-splits will contain two vertices with potentially large out-degree. Consider the in-split. Vertices 1 and 2 will each have half as many incoming edges as the vertex they replaced, possibly rounded up (if the number is odd), plus one more incoming edge from the other vertex. An equivalent argument can be made for the out-split. Then it is clear that each of the four vertices still requiring work has maximum in-degree or out-degree bounded above by $\frac{s}{2} + 1.5$ or $\frac{r}{2} + 1.5$. To simplify expressions, we will assume that $s \geq r$, and so in the worst case all four vertices have in-degree or out-degree $\frac{s}{2} + 1.5$. Of course, the arguments that proceed are identical if $s < r$, with all instances of $s$ replaced by $r$.

We then replace the four vertices with appropriate in-splits and out-splits. Call this process "Stage 2". Using the same argument as above, we see that the maximum in-degree or out-degree of the in-splits and out-splits used is $\frac{\frac{s}{2} + 1.5}{2} + 1.5$, and in each of the in-splits and out-splits there are at most two vertices that have in-degree or out-degree greater than $d$.

We can then continue the process of replacing high-degree vertices with in-splits and out-splits, and it is easy to check that at Stage $i$ there are $2^i$ new subgraphs introduced, and so there are at most $2^{i+1}$ vertices that still require work. Denote by $d_i$ the maximum in-degree or out-degree of vertices at Stage $i$. Then it is clear that $d_i \leq \frac{d_{i-1}}{2} + 1.5$, with $d_0 = s$. Solving this recursive relationship, we obtain that $d_{i} \leq \frac{s}{2^i} + 3$. Therefore, the maximum in-degree or out-degree after $n$ iterations, in the worst case, is something less than $\frac{s}{2^{n}} + 3$. We can then see how many iterations are necessary to reduce all vertices in the subgraph to maximum in-degree and out-degree to $d = 4$ or lower:

\begin{eqnarray*}\frac{s}{2^{n}} + 3 & \leq & d\\
\log_2(s) - (n) & \leq & \log_2(d-3)\\
n & \geq & \log_2(s) - \log_2(d-3)\end{eqnarray*}

So for all iterations $n \geq \log_2(s) - \log_2(d-3)$, the subgraph has all vertices with in-degree and out-degree $d$ or lower. So in the worst case, we will not need to consider more than $\log_2(s) + 1$ iterations.

We can now see how many vertices exist in the final subgraph. At each stage $1 \leq i \leq n-1$ we introduced $2^i$ new subgraphs (in-splits and out-splits), which each included a vertex that we did not need to further convert. Then, in stage $n$ there were $3(2^n)$ vertices introduced. Therefore the total number of vertices is at most:

$$\left[ \sum_{i=1}^{n-1} 2^i \right] + 3(2^n) = 2^{n+2} - 1 \leq \frac{8s}{d-3} - 1,$$

and therefore, the number of vertices in the final subgraph is of order $O(s)$. Since each vertex has bounded degree, it is clear that the number of edges is also $O(s)$.

Of course, as mentioned earlier, if $s < r$ then the above bounds are $O(r)$ instead, which leads to the result.\end{proof}

It is trivial to see that replacing a vertex with a subgraph, using the above procedure, does not alter the in-degree or out-degree of any other vertex in the graph. Then we may simply perform the above procedure for all vertices
in the graph with in-degree or out-degree greater than $d$. Suppose that, in the original instance containing $N$ vertices, vertex $i$ has maximum in-degree or out-degree $d_i$, and define $k := \sum_{i = 1}^N d_i$. It is clear that $k$ at least as big as the number of directed edges in the graph. Then, when the Splitting Procedure is applied to all vertices in the graph, there will be $O(k)$ vertices of in-degree and out-degree no larger than $d$.

Finally, we present the main theorem of this manuscript.

\begin{theorem}Using the procedures described in this manuscript, one may convert any general HCP instance to an equivalent (undirected) cubic HCP instance containing $O(k)$ vertices and $O(k)$ edges, where $k$ is defined as above.\end{theorem}

\begin{proof}The result follows immediately from Proposition \ref{prop-degree_d} and Corollary \ref{col-linear_conversion}. \end{proof}

Since the number of vertices in the original instance will definitely not be greater than the number of directed edges (or else the instance is trivially non-Hamiltonian), we can easily see that the input size of the original instance is of order $O(k)$. Then the upcoming procedure constitutes a linearly-growing conversion from general HCP to cubic HCP. In order to relate this result to the classical \lq\lq SAT to 3SAT" result of Karp \cite{karp}, which is also a linearly-growing conversion, we label the following as the \lq\lq HCP to 3HCP Conversion Procedure".

{\underline{HCP to 3HCP Conversion Procedure}}

\begin{itemize}\item Select a value of $d \geq 4$.
\item Perform the Splitting Procedure on any vertex with in-degree or out-degree greater than $d$.
\item While any vertices exist with combined in-degree and out-degree $s + r \geq 4$, replace them with an $(s,r)$-gate.
\item Convert the directed graph to an undirected graph.
\item Replace all degree 4 vertices with 4-gates using the Sub-quartic HCP to Sub-cubic HCP Conversion Procedure.
\item Convert the sub-cubic graph to a cubic graph using the Sub-cubic HCP to Cubic HCP Conversion Procedure.
\end{itemize}

Note that in the above procedure, we may instead choose to swap the third and fourth items. If so, the fifth step may be omitted.

In reality, the HCP to 3HCP Conversion Procedure as described above is needlessly heavy-handed. Although the proof of Proposition \ref{prop-degree_d} is only valid for $d \geq 4$, it can be checked that once the maximum in-degree or out-degree in a replacing subgraph has been reduced to 4, using one more iteration of in-splits and out-splits will reduce the maximum in-degree and out-degree to 3 (the lowest it may be reduced to). Clearly this one additional iteration of replacements can only double the size of the resultant graph, so the conversion remains linearly-growing. If the graph is then converted to an undirected graph at this point, the maximum degree will be 4, at which point 4-gates may be used. We summarise this alteration in the following procedure, which we call the \lq\lq Quick HCP to 3HCP Conversion Procedure".

{\underline{Quick HCP to 3HCP Conversion Procedure}}

\begin{itemize}\item Perform the Splitting Procedure on any vertex with in-degree or out-degree greater than $3$.
\item Convert the directed graph to an undirected graph.
\item Replace any degree 4 vertices with 4-gates using the Sub-quartic HCP to Sub-cubic HCP Conversion Procedure.
\item Convert the sub-cubic graph to a cubic graph using the Sub-cubic HCP to Cubic HCP Conversion Procedure.
\end{itemize}

We recommend that interested readers  use the Quick HCP to 3HCP Conversion rather than the original HCP to 3HCP Conversion Procedure, both for ease of implementation, and for the sake of the size of the resultant instance. We provide a bound on the size of instances produced by this conversion at the end of this section.

We now display, in Figures \ref{fig-example1}--\ref{fig-example4}, a vertex with in-degree 9 and out-degree 6, and the various intermediate steps as it is converted to an equivalent cubic subgraph by the Quick HCP to 3HCP Conversion Procedure. A path through the fourth incoming external edge and the fifth outgoing external edge is displayed at each stage. The final cubic subgraph contains 269 vertices.

\begin{figure}[h!]
\centering\hspace*{-0.5cm}\includegraphics[scale=0.45]{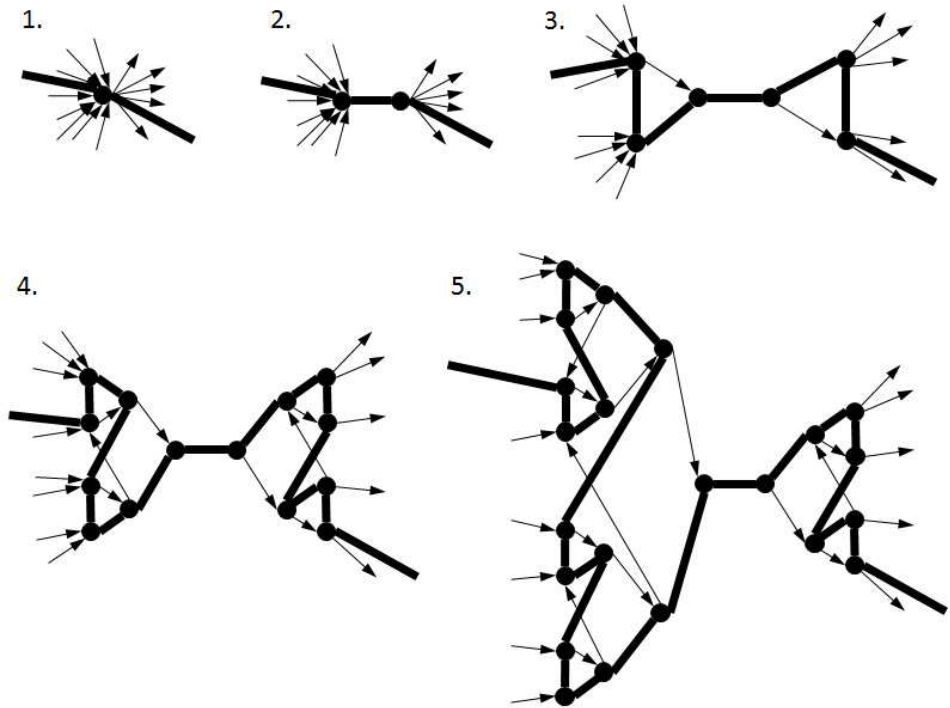}
\caption{A vertex with in-degree 9 and out-degree 6, and the subgraphs produced during the step 1 of the Quick HCP to 3HCP Conversion Procedure.\label{fig-example1}}
\end{figure}
\vspace*{1cm}
\begin{figure}[h!]
\centering\hspace*{-0.5cm}\includegraphics[scale=0.45]{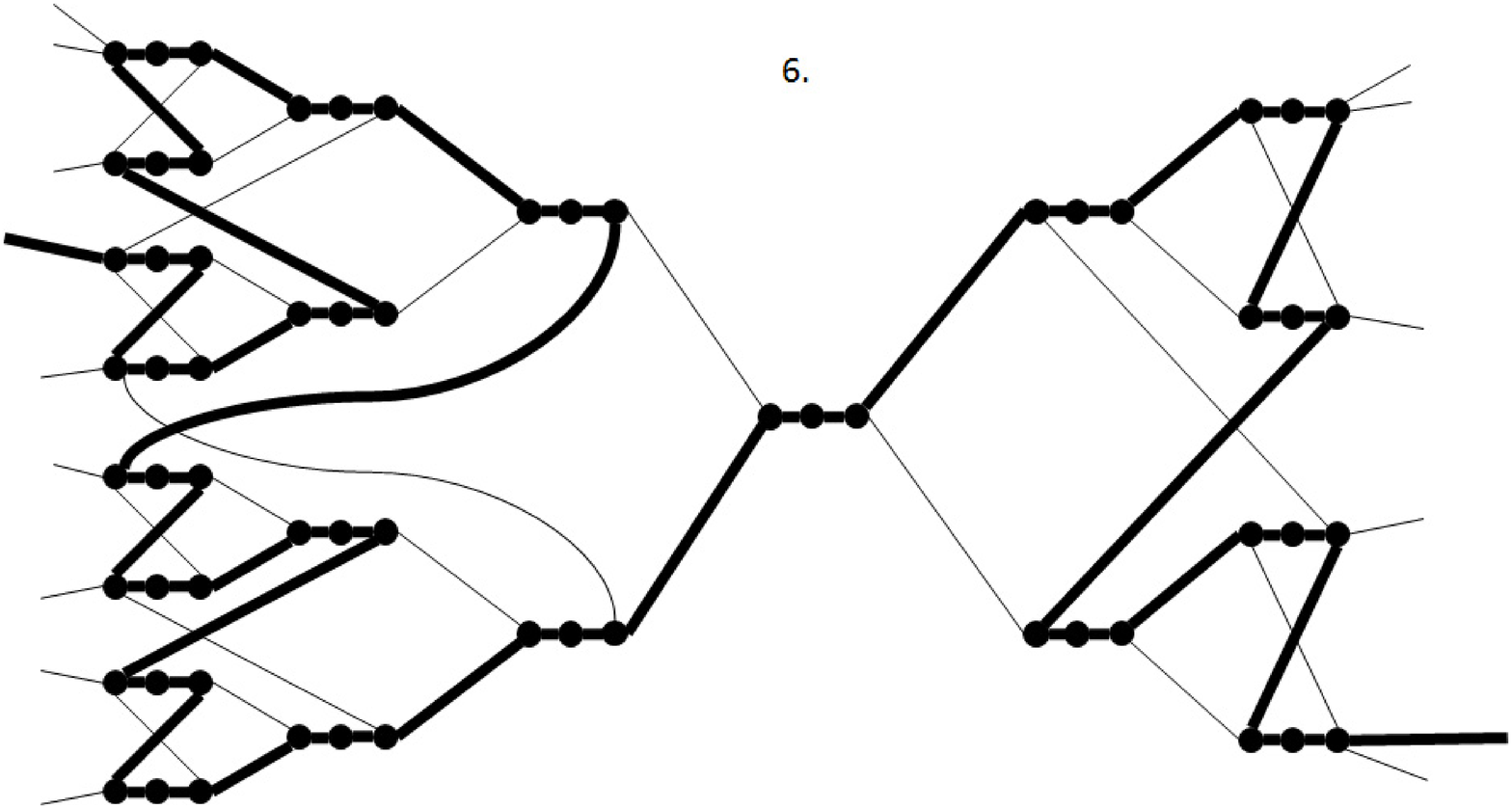}
\caption{The sub-quartic, undirected subgraph produced during step 2 of the Quick HCP to 3HCP Conversion Procedure.\label{fig-example2}}
\end{figure}
\newpage
\begin{figure}[h!]
\centering\hspace*{-0.5cm}\includegraphics[scale=0.52]{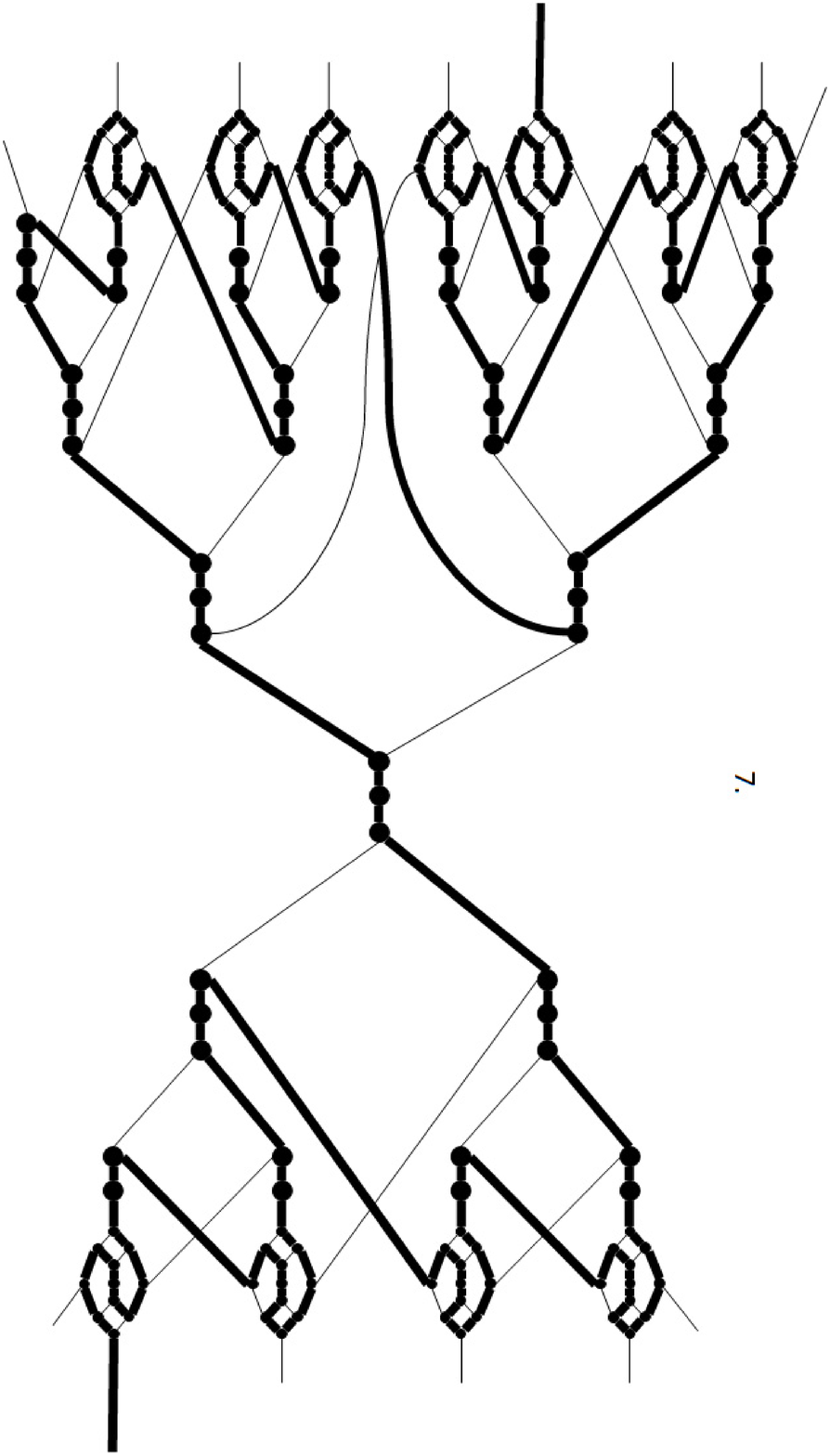}
\caption{The sub-cubic, undirected subgraph produced during step 3 of the Quick HCP to 3HCP Conversion Procedure.\label{fig-example3}}
\end{figure}
\newpage
\begin{figure}[h!]
\centering\hspace*{-0.5cm}\includegraphics[scale=0.52]{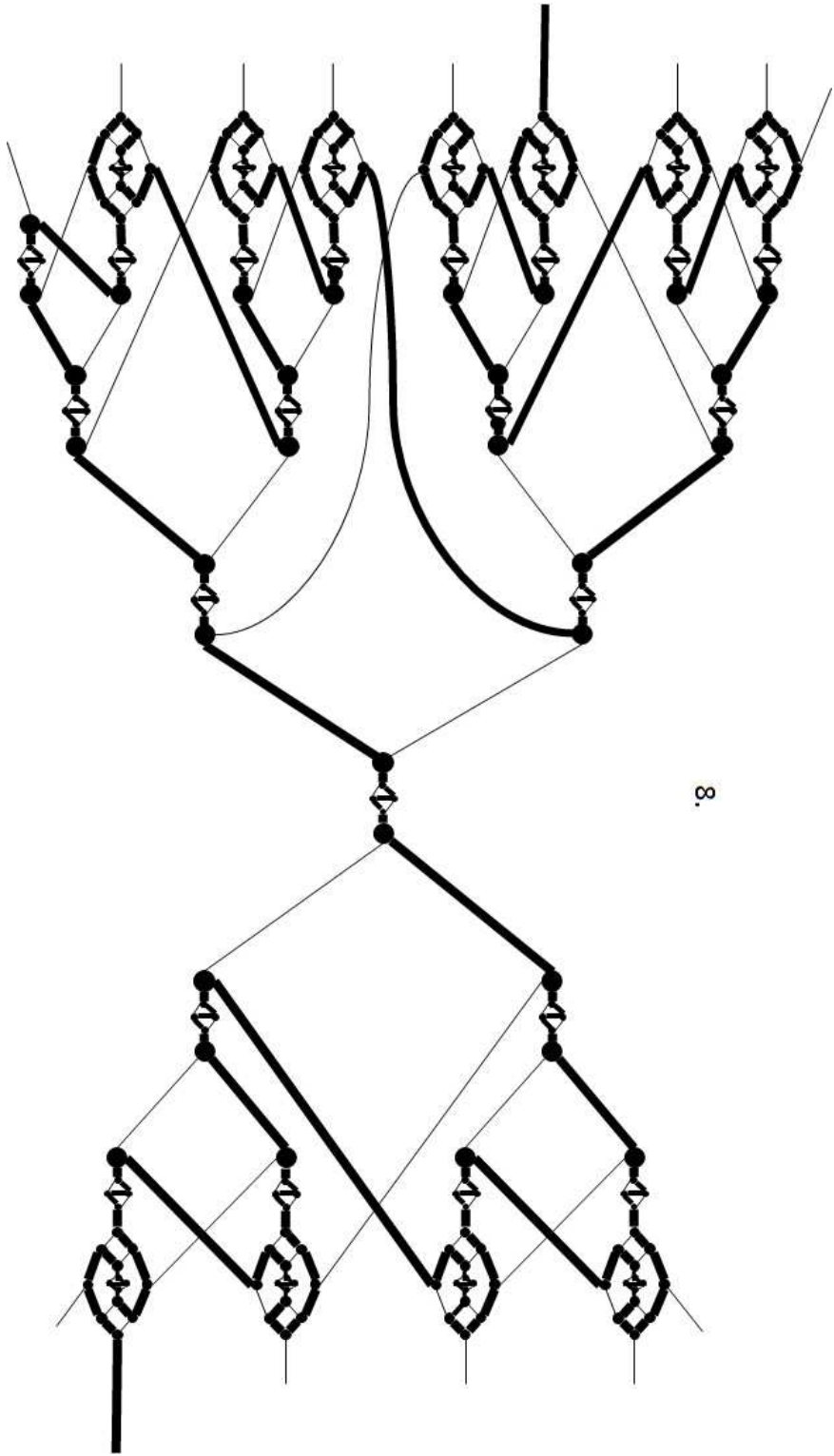}
\caption{The final, cubic, undirected subgraph produced by the Quick HCP to 3HCP Conversion Procedure.\label{fig-example4}}
\end{figure}
\newpage

We now conclude this section with a revised upper bound on the size of the cubic instance obtained from the Quick HCP to 3HCP Conversion Procedure.

\begin{lemma}Consider a graph $\Gamma$, and denote by $k$ the sum of in-degrees and out-degrees of all vertices in $\Gamma$. Then the instance obtained after performing the Quick HCP to 3HCP Conversion Procedure will contain no more than $25k$ vertices.\label{lem-final-size}\end{lemma}

\begin{proof}Consider a single vertex in $\Gamma$, with in-degree $s$ and out-degree $r$. It is easy to check that the subgraph which replaces this vertex after step 1 of the Quick HCP to 3HCP Conversion Procedure (e.g. see Figure \ref{fig-example1}) will contain no more than one vertex adjacent to each external edge. It can also be easily checked that there will be more vertices adjacent to external edges than those not adjacent. Therefore an upper bound on the number of vertices in the replacing subgraph is $2(s+r)$, and so, once step 1 is carried out for all vertices in $\Gamma$, the number of vertices in the resultant subgraph will be no more than $2k$. Note that at this stage, there are at most $k$ vertices with maximum in-degree or out-degree of 4.

In step 2, the number of vertices is tripled, so there are at most $6k$ vertices. However, there are still only at most $k$ vertices with maximum in-degree or out-degree 4. Also, at most $2k$ vertices have degree 2.

In step 3, we replace $k$ vertices with 4-gates, which each contain 11 vertices, including one vertex of degree 2. So at this point the number of vertices is no more than $16k$, of which at most $3k$ vertices have degree 2.

In step 4, we replace at most $3k$ vertices with diamonds, which each contain 4 vertices. So after the Quick HCP to 3HCP Conversion Procedure is completed, the graph contains at most $25k$ vertices.\end{proof}

For undirected graphs containing $e$ edges, it is trivial to see that $k = 4e$. The following corollary is this obvious, and does away with the need to count the in-degree and out-degree of each vertex.

\begin{corollary}If $\Gamma$ is an undirected graph containing $e$ undirected edges, the cubic graph resulting from the Quick HCP to 3HCP Conversion Procedure contains no more than $100e$ vertices and $150e$ edges.\label{cor-undirected}\end{corollary}

It should be noted that, for undirected graphs with moderate vertex degree (say, smaller than 50), it may be cheaper to use $s$-gates rather than the Quick HCP to 3HCP Conversion Procedure, despite the quadratic growth inherent in the former choice. This is partially due to the unnecessary conversion to a directed graph and then back again, after which the number of vertices is tripled. Even for directed graphs, if the vertex order is low (say, smaller than 15) it can still be more efficient to use $(s,r)$-gates. The interested reader should investigate both approaches for their problem to see which works better.

\section{Improvement over the existing approach}

We have implemented the Quick HCP to 3HCP Conversion Procedure using Java and have converted several famous graphs to cubic graphs using our approach. We now compare the corresponding instance input sizes with those of the cubically-growing conversion obtained by first converting from HCP to SAT, then to 3SAT, and finally to 3HCP. We convert to SAT using the Ariadne100 software package \cite{ariadne}, then to 3SAT using the conversion given in Karp \cite{karp}, and finally to 3HCP using the approach by Garey et al \cite{tarjan}. Note that, since we do not demand planarity, we may avoid the most expensive step in their approach. In fact, although it is not stated explicitly in \cite{tarjan}, it can be relatively easily checked that if the 3SAT instance has $m$ clauses and $n$ literals, the resultant cubic graph will contain $408m + 40n$ vertices. It is also interesting to note that Ariadne100 produces SAT instances for which the number of vertices bounds the size, rather than edges. Specifically, if a graph has $N$ vertices and $e$ undirected edges, the SAT instance has $2N^3 - 2N^2 + 2N - 2Nm$ and $N^2$ literals. Note that this implies that dense graphs will actually produce smaller SAT instances than more sparse graphs of the same number of vertices. This is in contrast to our result, where the size of the resultant graph instance grows proportionally to the number of edges present in the original graph instance. The comparative results are listed in Table \ref{tab-results}. In the first column of Table \ref{tab-results} we list the instances considered. In the second column we list the number of vertices in the graph. In the third column we list the maximum and average degree of the vertices in the graph. In the fourth column, we list the number of vertices in the resultant cubic graph after applying the Quick HCP to 3HCP Conversion Procedure. In the fifth column, we list the number of vertices in the resultant cubic graph after replacing all vertices with appropriate $s$-gates iteratively until a cubic graph is obtained. In the sixth column, we list the number of vertices in the resultant cubic graph after the alternative approach using the result of Garey et al. We have chosen twelve well-known undirected graphs to consider as instances of HCP, which we describe after Table \ref{tab-results}.

\begin{table}[h!]\hspace*{-0.35cm}\begin{tabular}{|c|c|c|c|c|c|}\hline
{\bf Graph} & {\bf Vertices} & {\bf (Max, mean)} & {\bf Vertices in} & {\bf Vertices in} & {\bf Vertices in}\\
& & {\bf Degree} & {\bf HCP to 3HCP} & {\bf $s$-gate} & {\bf Garey et al}\\
\hline
\hline
$\mbox{K}_{10}$  & 10 & (9,9) & 3,560 & 1,090 & 845,280 \\
\hline
Goldner-Harary  & 11 & (8,4.9091) & 1,594 & 336 & 1,655,720 \\
\hline
Sousselier  & 16 & (5,3.375) & 932 & 96 & 6,044,160 \\
\hline
6-Andr\'{a}sfai  & 17 & (6, 6) & 3,502 & 782 & 6,670,664 \\
\hline
24-cell  & 24 & (8,8) & 7,344 & 2,064 & 19,230,336 \\
\hline
29-Paley  & 29 & (14,14) & 17,574 & 7,366 & 30,968,520 \\
\hline
Foster Cage  & 30 & (5,5) & 4,680 & 870 & 41,617,440 \\
\hline
Sheehan-40  & 40 & (39,20.05) & 36,316 & 24,784 & 80,791,040 \\
\hline
Sims-Gewirtz  & 56 & (10,10) & 22,736 & 7,504 & 271,272,064 \\
\hline
8x8 Knight's Tour  & 64 & (8,5.25) & 10,592 & 2,416 & 427,084,800 \\
\hline
Sheehan-80  & 80 & (79,40.025) & 152,556 & 186,324 & 652,154,880 \\
\hline
$\mbox{K}_{100}$  & 100 & (99,99) & 485,600 & 1,036,900 & 856,612,800 \\
\hline\end{tabular}
\caption{Comparative sizes of cubic graphs obtained from the conversions of some famous higher-degree graphs. The HCP to 3HCP conversion and $s$-gate conversion are those described in this manuscript. The Garey et al conversion is that which uses (in the final step) the result by Garey et al \cite{tarjan}.\label{tab-results}}\end{table}

Table \ref{tab-results} indicates very clearly the savings that may be obtained from using the approaches described in this manuscript. Note here that the examples with smaller degree actually result in the quadratically-growing conversion outperforming the linearly-growing conversion. For the examples of larger degree, however, the linearly-growing conversion is superior. Both approaches dramatically outperform the approach of first converting to SAT, then to 3SAT, and finally back to 3HCP. The twelve famous graph instances considered are:

\begin{itemize}\item The complete graph on 10 vertices, $\mbox{K}_{10}$.
\item Goldner-Harary Graph \cite{goldner}, the smallest non-Hamiltonian maximal planar graph.
\item Sousselier Graph, one of four hypohamiltonian graph on 16 vertices (see Aldred et al \cite{aldred}).
\item 6-Andr\'{a}sfai Graph \cite{godsil}, a circulant graph in which each vertex $i$ is adjacent to vertices $i + j$ for all $j$ congruent to $1\mod 3$.
\item 24-cell, also known as the icositetrachoron, the unique self-dual Euclidian polytope that is neither a polygon nor a simplex (see Coxeter \cite{coxeter}).
\item Paley Graph \cite{erdos} on 29 vertices, one of an infinite family of conference graphs.
\item Foster Cage, one of four (5,5)-cages (see Meringer \cite{meringer}, or Wong \cite{wong} who listed the Foster cage but missed one of the other (5,5)-cages).
\item Sheehan-40 Graph, one of an infinite family of maximal uniquely Hamiltonian graphs (see Sheehan \cite{sheehan}).
\item Sims-Gewirtz Graph \cite{gewirtz}, an integral 10--regular graph with the unique graph spectrum $10^12^{35}(-4)^{20}$.
\item 8x8 Knight's Tour Graph, in which each vertex corresponds to a square on the chessboard, and with edges present only when one square can be reached from another by the move of a knight (see Conrad et al \cite{conrad}).
\item Sheehan-80 Graph, another of the infinite family described by Sheehan \cite{sheehan}.
\item The complete graph on 100 vertices, $\mbox{K}_{100}$.
\end{itemize}

It should be noted that the intention of Table \ref{tab-results} is not to disparage the result in Garey et al, but rather to indicate the folly of switching problem frameworks unnecessarily, or alternatively, the desirability of developing problem-specific conversions rather than relying on those that exist via other NP-complete problems. We conclude this manuscript with a short discussion on this topic.

\section*{Final Thoughts and Future Work}

In the years since the Cook-Levin theorem \cite{cook} was published, and Karp \cite{karp} listed 21 NP-complete problems with SAT as a base, it has become commonplace for researchers to think of SAT as the best choice of problem to use for conversions. Indeed, other than perhaps 0-1 integer programming, the majority of algorithmic development for solving NP-complete problems appears to have been focused on SAT. However, there is nothing inherently special about SAT to warrant this special consideration among other NP-complete problems.

Conversely, there has been relatively little attention paid to the degree of the polynomial that arises in conversions of one NP-complete problem to another. As long as the conversion is polynomially-growing, it is deemed to be efficient. However, in practice, even if a polynomial-time algorithm could be found for some special NP-complete problem, other problems of real-world size would still be intractable if they first require a conversion that grows the input size by a polynomial of even moderate degree (even two or three). Even extremely fast heuristics are practically useless for real-world sized problems if the input size has been blown up by anything larger than linear growth. Perhaps the first specific consideration of linearly-growing conversions was by Dewdney \cite{dewdney} who showed that six problems have linearly-growing conversions to SAT and from SAT. Creignou \cite{creignou} extended this idea by developing the concept of {\em SAT-easy} and {\em SAT-hard} problems. She described SAT-easy problems as those that are reducible to Satisfiability with linear growth in the input size, and SAT-hard problems as those that may be obtained by a reduction from Satisfiability with linear growth in the input size. If a problem is both SAT-easy and SAT-hard, Creignou calls it {\em SAT-equivalent}. In her paper, Creignou described 14 NP-complete problems which are SAT-equivalent, but identified only two NP-complete problems which are SAT-hard but not SAT-easy, one of which is HCP (the other is the disjoint connecting paths problem). Creignou also suggested that the 3-colorability problem (3COL) might be the most natural NP-complete problem, and provided a methodology of constructing reductions to 3COL for many NP-complete problems.

Motivated by the result of the present manuscript, we hope to investigate the equivalent notion of {\em HCP-easy}, {\em HCP-hard} and {\em HCP-equivalent} problems. Certainly it is clear, from the result in \cite{creignou}, that any problem which is SAT-easy must also be HCP-easy, however the converse is not known to be true. Likewise, a problem which is HCP-hard must also be SAT-hard. The result in the present manuscript indicates that 3HCP is HCP-equivalent but is not known to be SAT-equivalent. The study of which problems are HCP-equivalent but not SAT-equivalent, and vice versa, is a topic for future research.

Research is also warranted on the question of which problems may be converted, with only linear growth, directly to HCP, rather than via SAT. Although there are many NP-complete problems known to have linearly-growing conversions to SAT, and therefore to HCP, it seems inefficient to convert via another problem if it is not necessary to do so. A recent example of such research \cite{setsplitting} demonstrated that the Set Splitting problem, long known to have a linearly-growing conversion to SAT, also has a linearly-growing conversion directly to HCP with very small coefficients.

Ultimately, it seems natural to try to identify what might be thought of as a {\em minimal set} of NP-complete problems, being the smallest possible set such that every other NP-complete problem may be converted to one in the minimal set with only a linear growth in the input size.  Determining which problems may lie in this minimal set seems a fascinating topic which, to date, does not appear to have been considered at all. With computing power having improved to the stage where clever heuristics are now a viable method of solving many real-world sized NP-complete problem instances, the time seems right for such exploration to occur.

\section*{Acknowledgments}

The authors gratefully acknowledge useful conversations with Jerzy Filar and Pouya Baniasadi. The research presented in this manuscript was supported in part by the DSTO, and the ARC Discovery Grant DP120100532.


\begin{thebibliography}{99}

\bibitem{aldred} R. E. L. Aldred, B. D. McKay, N. C. Wormald, {\em Small Hypohamiltonian Graphs}, J. Combin. Math. Combin. Comput., 23:143--152, 1997.
\bibitem{barnette} D. W. Barnette, {\em Conjecture 5}, in Tutte, W. T., {\em Recent Progress in Combinatorics: Proceedings of the Third Waterloo Conference on Combinatorics, May 1968}, New York: Academic Press, 1969.
\bibitem{conrad} A. Conrad, T. Hindrichs, H. Morsy, I. Wegener, {\em Solution of the Knight's Hamiltonian Path Problem on Chessboards}, Discrete Appl. Math., 50(2):125--134, 1994.
\bibitem{cook} S. Cook, {\em The complexity of theorem proving procedures}, {\em Proceedings of the Third Annual ACM Symposium on Theory of Computing}, New York, pp. 151--158, 1971.
\bibitem{coxeter} H. S. M. Coxeter, {\em Introduction to Geometry, 2nd ed.} New York: Wiley, p. 404, 1969.
\bibitem{creignou} N. Creignou, {\em The class of problems that are linearly equivalent to Satisfiability or a uniform method for proving NP-completeness}, Theor. Comput. Sci., 145:111-145, 1995.
\bibitem{dewdney} A. K. Dewdney, {\em Linear transformations between combinatorial problems}, Internat. J. Comput. Math., (11):91--110, 1982.
\bibitem{eppstein} D. Eppstein, {\em The traveling salesman problem for cubic graphs}, J. Graph Algorithms Appl., 11(1):61--81, 2007.
\bibitem{erdos} P. Erd\H{o}s, A. R\'{e}nyi, {\em Asymmetric graphs}, Acta. Math. Acad. Sci. Hung., 14(3-4):295--315, 1963.
\bibitem{setsplitting} J. A. Filar and M. Haythorpe, {\em A Linearly-Growing Conversion from the Set Splitting Problem to the Directed Hamiltonian Cycle Problem}, in: {\em Optimization and Control methods in Industrial Engineering and Construction}, to appear.
\bibitem{conjecturepaper} J.A. Filar, M. Haythorpe and G.T. Nguyen. A conjecture on the prevalence of cubic bridge graphs. {\em Discuss. Mat. Graph Th.}, 30(1):175--179 (2010)
\bibitem{gareyjohnson} M. R. Garey and D. S. Johnson, {\em Computers and intractiability: A Guide to the Theory of NP-Completeness}, W. H. Freeman, 1979.
\bibitem{tarjan} M. R. Garey, D. S. Johnson, R. E. Tarjan, {\em The planar Hamiltonian circuit problem is NP-complete}, SIAM J. Comput., 5(4):704--714, 1976.
\bibitem{gebauer} H. Gebauer, {\em On the number of Hamilton cycles in bounded degree graphs}, in: {\em Proceedings 4th Workshop on Analytic Algorithms and Combinatorics (ANALCO 08)}, 2008.
\bibitem{gewirtz} A. Gewirtz, {\em The Uniqueness of $g(2, 2, 10, 56)$}, Trans. New York Acad. Sci., 31:656--675, 1969.
\bibitem{godsil} C. Godsil, and G. Royle, {\em The Andr\'{a}sfai Graphs}, in: {\em Algebraic Graph Theory}, New York: Spinger-Verlag, pp. 118--123, 2001.
\bibitem{goldner} A. Goldner, F. Harary, {\em Note on a smallest nonhamiltonian maximal planar graph}, Bull. Malaysian Math. Soc., 6(1):41--42, 1975.
\bibitem{holton} D. A. Holton and J. Sheehan, {\em The Petersen Graph}, Cambridge University Press, 1993.
\bibitem{karp} R. M. Karp, {\em Reducibility among combinatorial problems}, Springer, New York, 1972.
\bibitem{meringer} M. Meringer, {\em Fast Generation of Regular Graphs and Construction of Cages}, J. Graph Th., 30:137--146, 1999.
\bibitem{ariadne} M. Nasu, {\em Ariadne100: Hamiltonian Circuit Experiment Project}, cited 7th May 2013, \textcolor{blue}{\underline{http://www.aya.or.jp/$\sim$babalabo/Ariadne/Ariadne.html}}, 2000.
\bibitem{sheehan} J. Sheehan, {\em Graphs with exactly one hamiltonian circuit}, J. Graph Th., 1:37--43, 1977.
\bibitem{tutte} W. T. Tutte, {\em On Hamiltonian circuits}, J. Math. Soc., 21:98--101, 1946.
\bibitem{wong} P. K. Wong, {\em Cages -- A Survey}, J. Graph Th., 6:1--22, 1982.

\end{thebibliography}
\end{document}